\documentclass[sn-mathphys]{sn-jnl}

\usepackage{natbib}
\usepackage{amsmath}
\usepackage{amsfonts}
\usepackage{lastpage}
\usepackage{graphicx}
\usepackage{subfigure}
\usepackage{multirow}
\usepackage{tabu}
\usepackage{lipsum}
\usepackage{essay-def-springer}
\usepackage{caption}
\usepackage{mathrsfs}
\usepackage[algo2e]{algorithm2e}
\usepackage{mathtools}
\usepackage{stackengine}
\setstackEOL{\\}
\usepackage{pdfpages}
\usepackage{multirow}

\jyear{2022}%

\theoremstyle{thmstyleone}%
\newtheorem{theorem}{Theorem}
\newtheorem{proposition}[theorem]{Proposition}%

\theoremstyle{thmstyletwo}%
\newtheorem{remark}{Remark}%

\theoremstyle{thmstylethree}%

\newcommand{\figsize}{0.4}

\begin{document}

\title[Sketching the Krylov Subspace]{Sketching the Krylov Subspace: Faster Computation of the Entire Ridge Regularization Path}


\author[1]{\fnm{Yifei} \sur{Wang}}\email{wangyf18@stanford.edu}

\author[1]{\fnm{Mert} \sur{Pilanci}}\email{pilanci@stanford.edu}

\affil[1]{\orgdiv{Department of Electrical Engineering}, \orgname{Stanford University}, \orgaddress{\street{350 Serra Mall}, \city{Stanford}, \postcode{94305}, \state{CA}, \country{USA}}}


\abstract{We propose a fast algorithm for computing the entire ridge regression regularization path in nearly linear time. Our method constructs a basis on which the solution of ridge regression can be computed instantly for any value of the regularization parameter. Consequently, linear models can be tuned via cross-validation or other risk estimation strategies with substantially better efficiency. The algorithm is based on iteratively sketching the Krylov subspace with a binomial decomposition over the regularization path. We provide a convergence analysis with various sketching matrices and show that it improves the state-of-the-art computational complexity. We also provide a technique to adaptively estimate the sketching dimension. This algorithm works for both the over-determined and under-determined problems. We also provide an extension for matrix-valued ridge regression. The numerical results on real medium and large scale ridge regression tasks illustrate the effectiveness of the proposed method compared to standard baselines which require super-linear computational time.}




\maketitle

\section{Introduction}
We consider the following ridge regression problem
\begin{equation}\label{least-square}
\min_{x\in \mbR^d} \frac{1}{2}\|Ax-b\|_2^2+\frac{\lambda}{2}\|x\|_2^2,
\end{equation}
where $A\in \mbR^{n\times d}$ is the data matrix, $b\in \mbR^d$ is the label vector and $\lambda>0$ is a regularization parameter. 

Typically, one needs to estimate the regularization parameter $\lambda$ from a set $\Lambda$ of possible values and select the $\lambda$ with the best performance on the validation set. For moderate size problem, to obtain an estimate for the regularization parameter $\lambda$, one can apply risk estimators such as generalized cross-validation \cite{gcvaa}, Stein’s unbiased risk estimate \cite{fsure}, or unbiased prediction risk estimate \cite{cmfib}. Moreover, solving this problem efficiently on a large scale is of great interest in model selection \cite{msikr} and transfer learning \cite{asodt}. For instance, in deep learning-based transfer learning, one needs to tune the last layer of a trained neural network to adapt to a different dataset. Essentially, training the last layer of the neural network is a ridge regression problem using squared loss, where the previous layers of the neural network can be fixed as feature extractors of the raw data. 

Classical methods for solving ridge regression include Singular Value Decomposition (SVD) method,  warm-started conjugate gradient (CG) method, warm-started preconditioned conjugate gradient (PCG) method, and warm-started iterative Hessian (IHS) sketch method \citep{ihs}. The SVD method constructs a decomposition of the data matrix and provides a closed-form parameterization of the optimal solution to \eqref{least-square} as a function of $\lambda$ (see e.g. \citep{frr}). The warm-started CG/PCG/IHS method iteratively solves \eqref{least-square} with different values of the regularization parameter $\lambda$. They leverage previous solutions along the regularization path as initializers to warm-start the iterations. Besides, for kernel ridge regression, Nystr\"om computational regularization (NCR) \citep{limnc} applies the Nystr\"om subsampling approaches to reduce the computation cost for calculating the regularization path.  

\begin{table}[htbp]
    \centering

    \begin{tabular}{|c|c|c|}
    \hline
         Method& small $T$&large $T$\\\hline
         IHS-BIN (Ours)&$O\pp{d_e^2d+(dd_e+\text{nnz}(A))\log\pp{\frac{\lambda_\text{max}}{\lambda_\text{min}}}}$&$O(Td)$\\\hline
         SVD&$O(nd^2)$&$O(Td^2)$\\\hline
         warm-started CG&$O(T\text{nnz}(A)\sqrt{\kappa})$&$O(T\text{nnz}(A)\sqrt{\kappa})$\\\hline
         warm-started PCG&$O(d_e^2d+\log(d_e)\text{nnz}(A))$&$O(T\text{nnz}(A))$\\\hline
         warm-started IHS&$O(d_e^2d+\log(d_e)\text{nnz}(A))$&$O(T\text{nnz}(A))$\\\hline
         NCR&$O(d_e^2n)$&$O(Td_e^3)$\\\hline
          \end{tabular}
          
    \caption{Computational complexity of solving the ridge regression problem \eqref{least-square} for $T$ distinct values of the regularization parameter $\lambda$. Here $\kappa$ is the condition number of $A+\lambda_\text{min} I$ and $d_e$ is effective dimension of $A+\lambda_\text{min} I$. }
    \label{tab:comp}
\end{table}


In this paper, we present the iterative Hessian sketch method with binomial decomposition (IHS-BIN) for rapidly solving ridge regression with multiple regularization parameters. The idea is to approximate the linear opearator $(A^TA+\lambda I)^{-1}$ via a polynomial of $\lambda$ constructed by the iterative Hessian sketch method (IHS) \cite{ihs}. To be specific, IHS is an efficient randomized algorithm for solving large-scale least-square problems. We first focus on the overdetermined case, where $n>d$, and then introduce an extension to the underdetermined case $n\le d$. Suppose that the size of $\Lambda$ is $T$ and $\Lambda\subseteq [\lambda_\text{min},\lambda_\text{max}]$ with $0<\lambda_\text{min}<\lambda_\text{max}$. We compare the computation cost of the proposed IHS-BIN method with other classical solvers for ridge regression in Table \ref{tab:comp}. For the case where $T$ is large, IHS-BIN with the computation cost of $O(Td)$ is the fastest solver to the best of our knowledge. When $T$ is small, IHS-BIN still offers a substantial improvement in complexity. In Figure \ref{fig:random_od} and \ref{fig:cifar}, we present the results on a randomly generated data example and a matrix-valued ridge regression problem with kernel matrix based on the CIFAR-10 dataset. We can observe that IHS-BIN can be significantly faster than other solvers.

\begin{figure}[ht]
\centering
\begin{minipage}[t]{\figsize\textwidth}
\centering
\includegraphics[width=\linewidth]{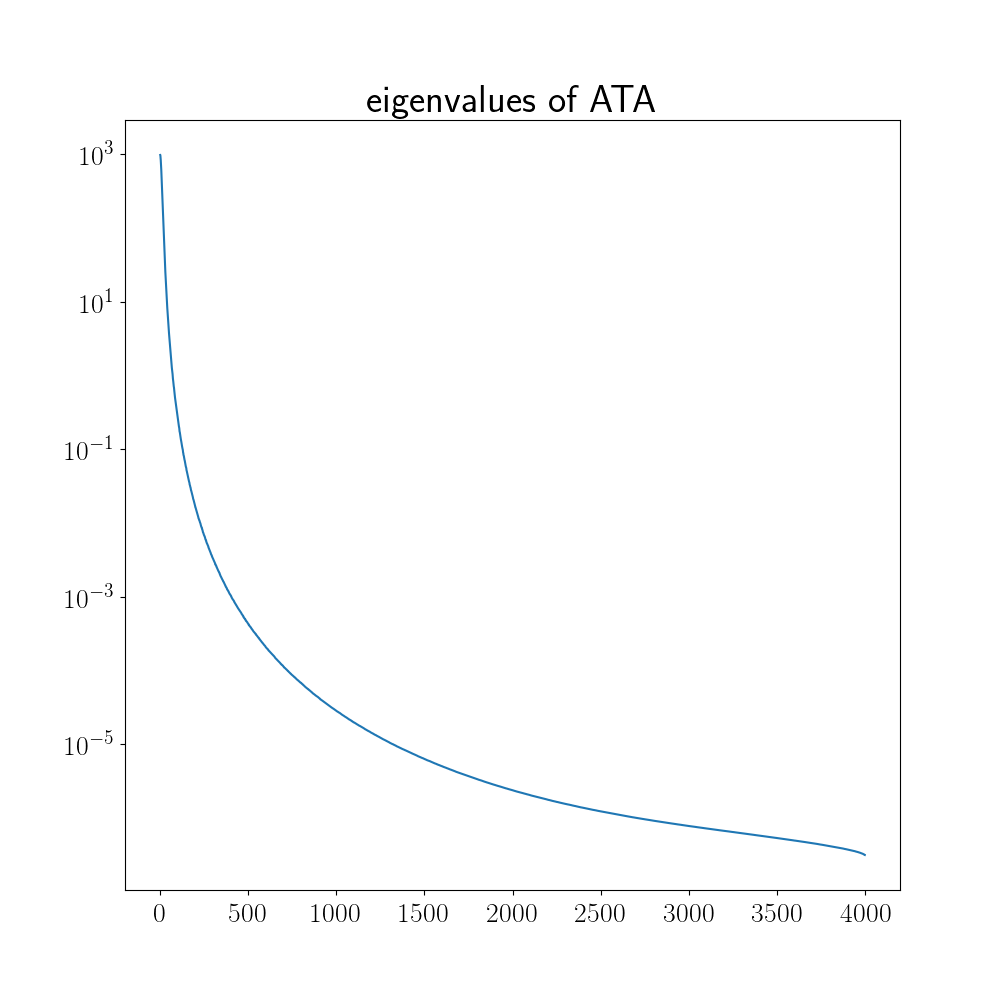}
\end{minipage}
\begin{minipage}[t]{\figsize\textwidth}
\centering
\includegraphics[width=\linewidth]{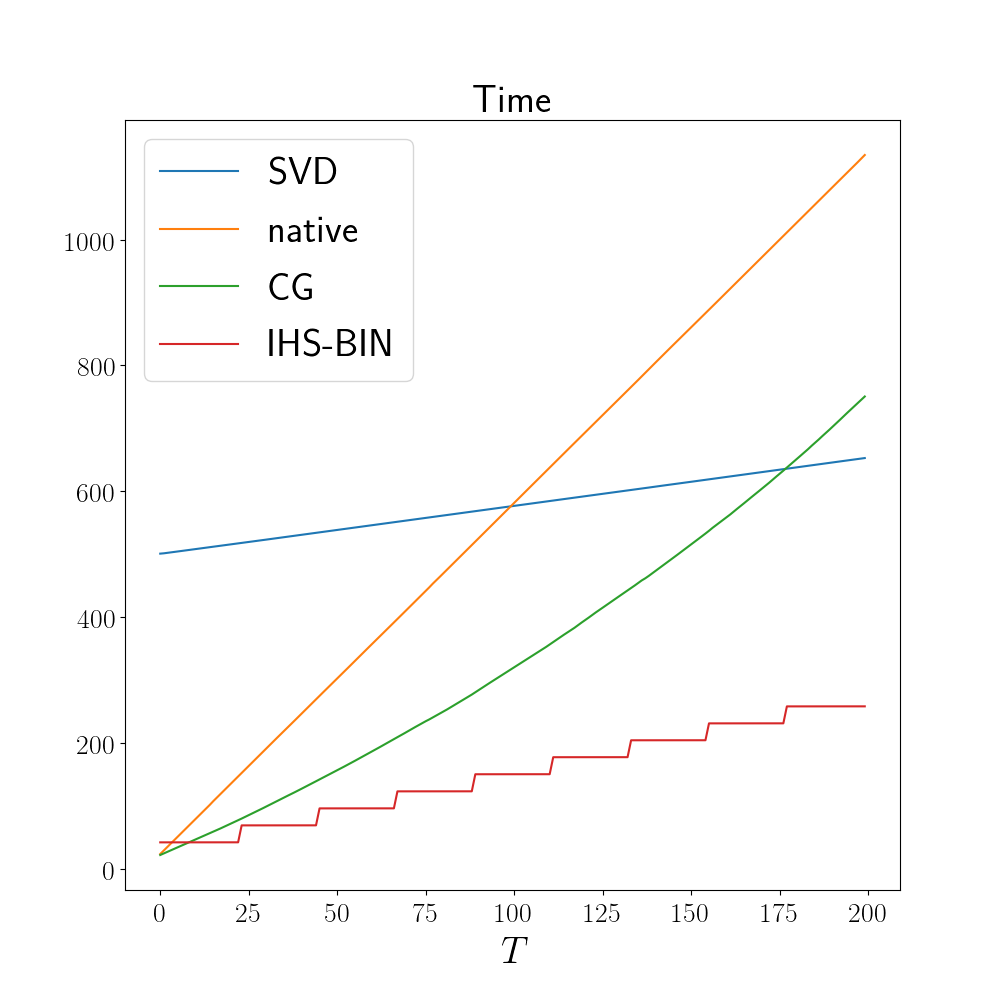}
\end{minipage}
\begin{minipage}[t]{\figsize\textwidth}
\centering
\includegraphics[width=\linewidth]{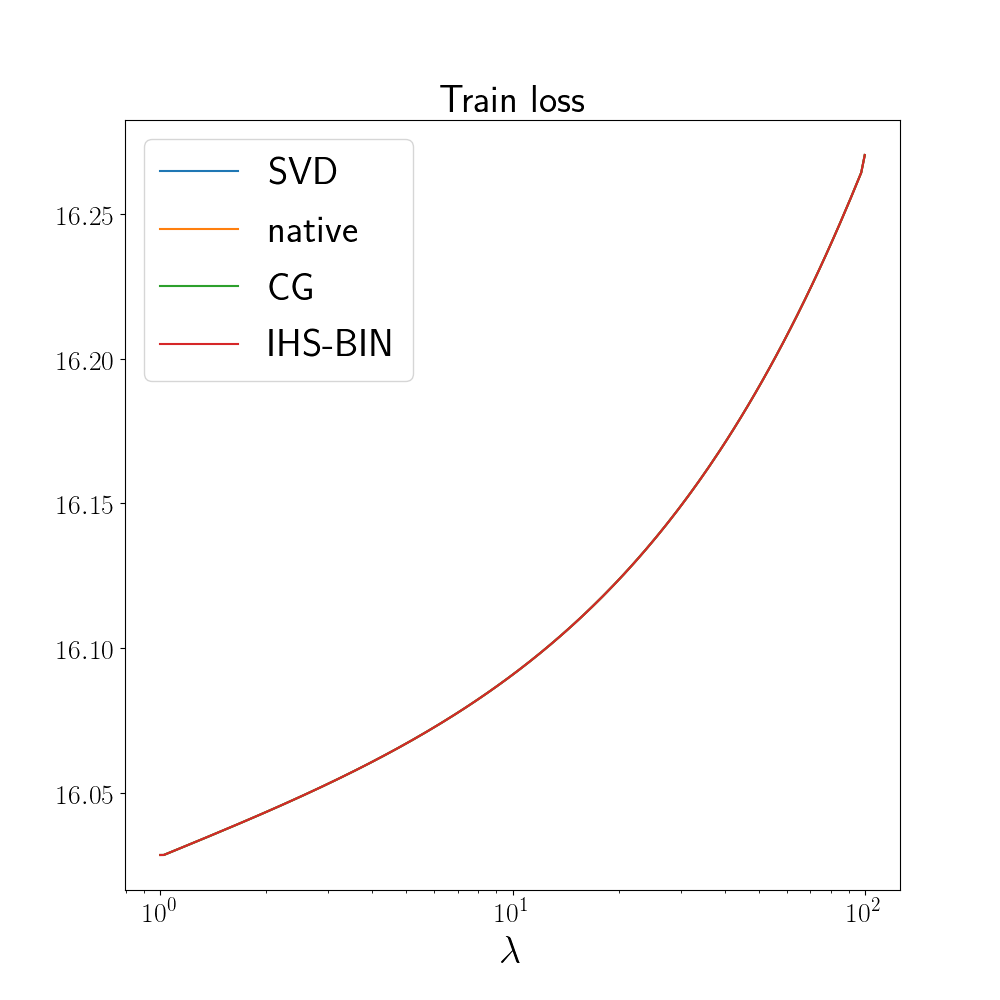}
\end{minipage}
\begin{minipage}[t]{\figsize\textwidth}
\centering
\includegraphics[width=\linewidth]{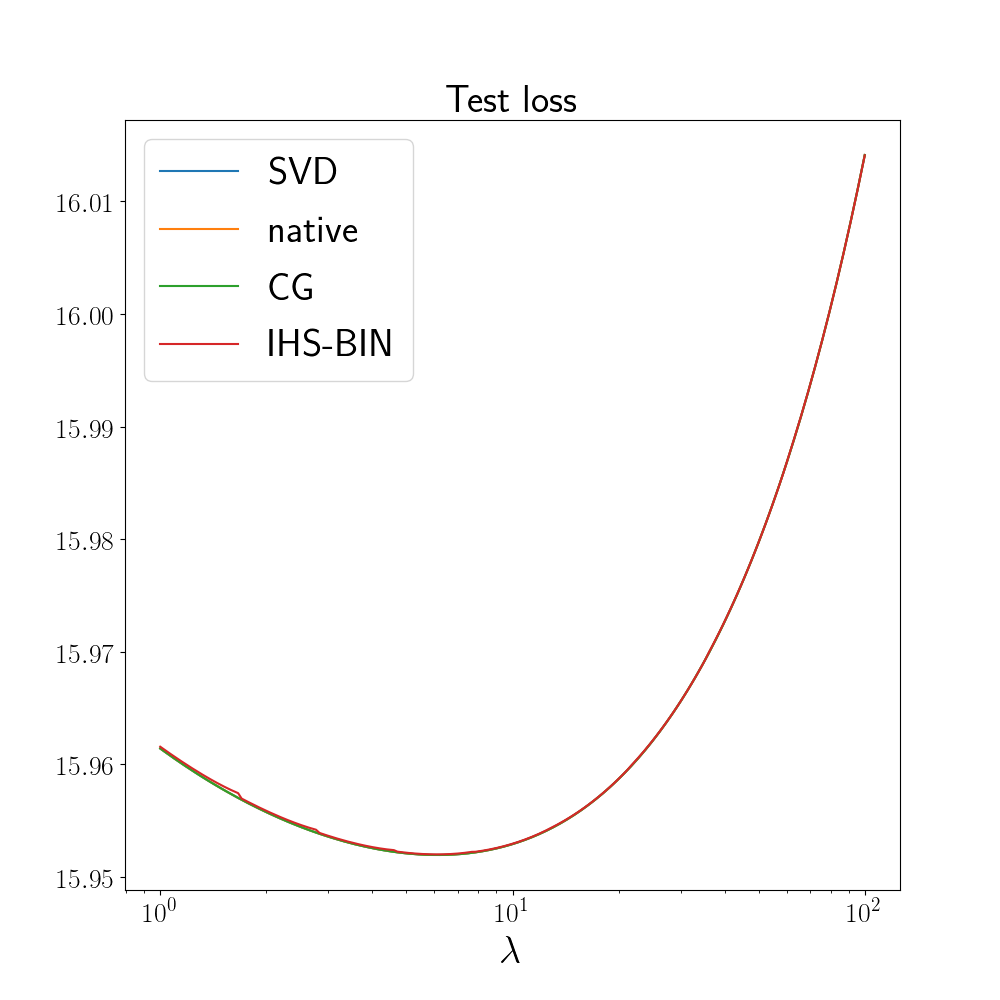}
\end{minipage}
\caption{Training loss, test loss and CPU time on randomly generated data. $n=20000, d=4000$, $m=1600$. $\lambda_\text{min}=1$. $\lambda_\text{max}=100$. `native' is the native linear system solver in NumPy.  } 
\label{fig:random_od}
\end{figure}

\begin{figure}[!htbp]
\centering
\begin{minipage}[t]{\figsize\textwidth}
\centering
\includegraphics[width=\linewidth]{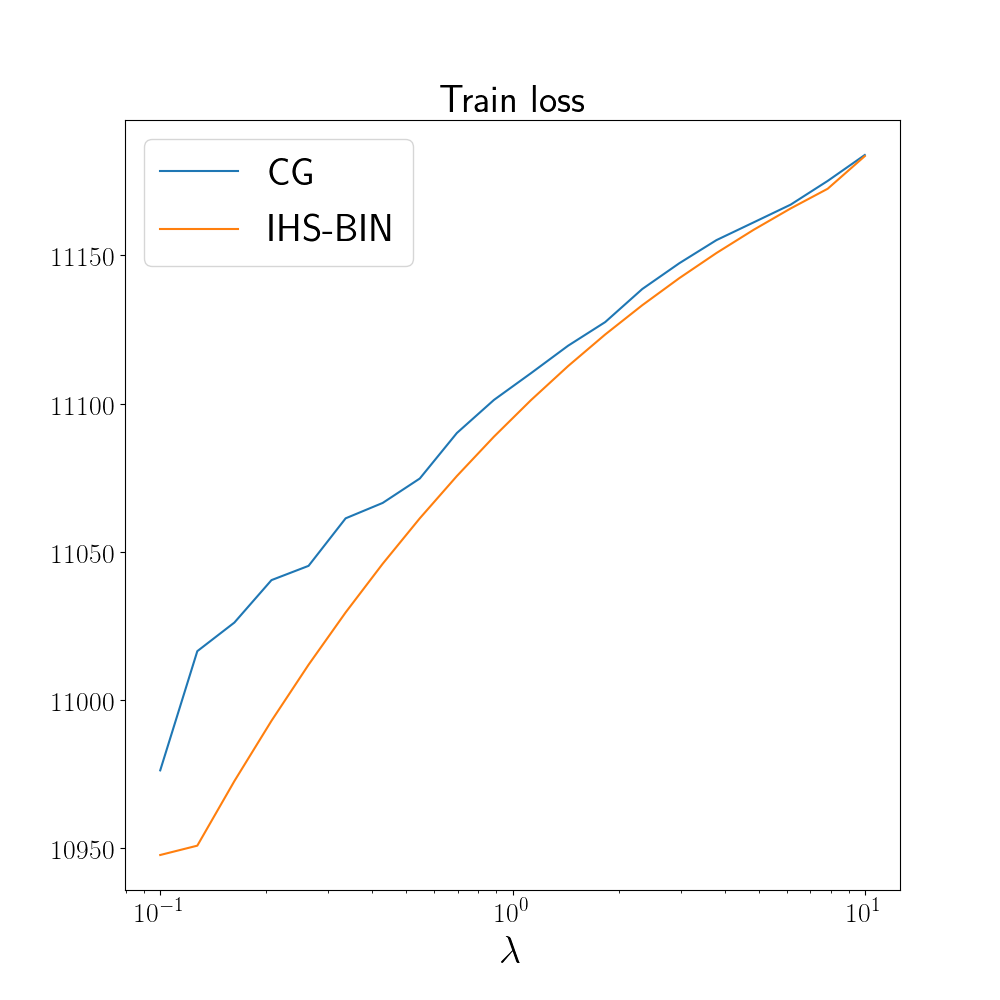}
\end{minipage}
\begin{minipage}[t]{\figsize\textwidth}
\centering
\includegraphics[width=\linewidth]{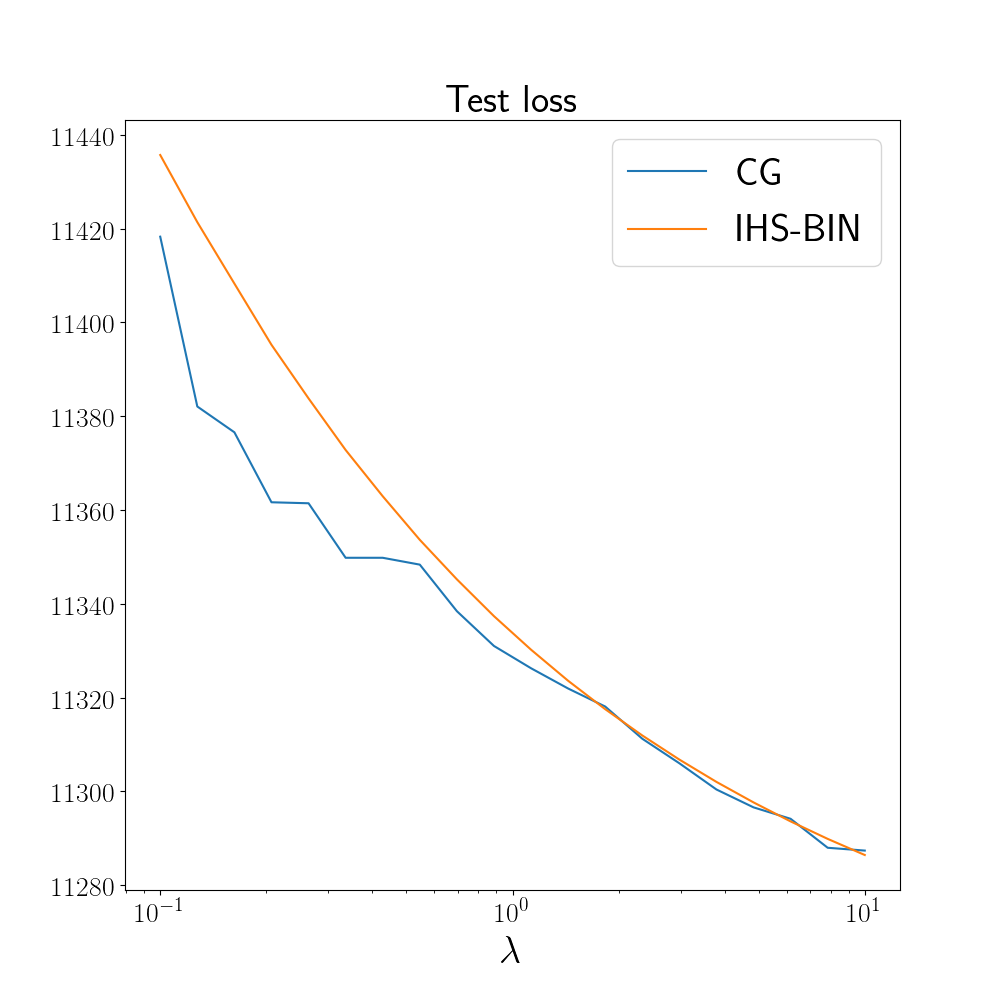}
\end{minipage}
\begin{minipage}[t]{\figsize\textwidth}
\centering
\includegraphics[width=\linewidth]{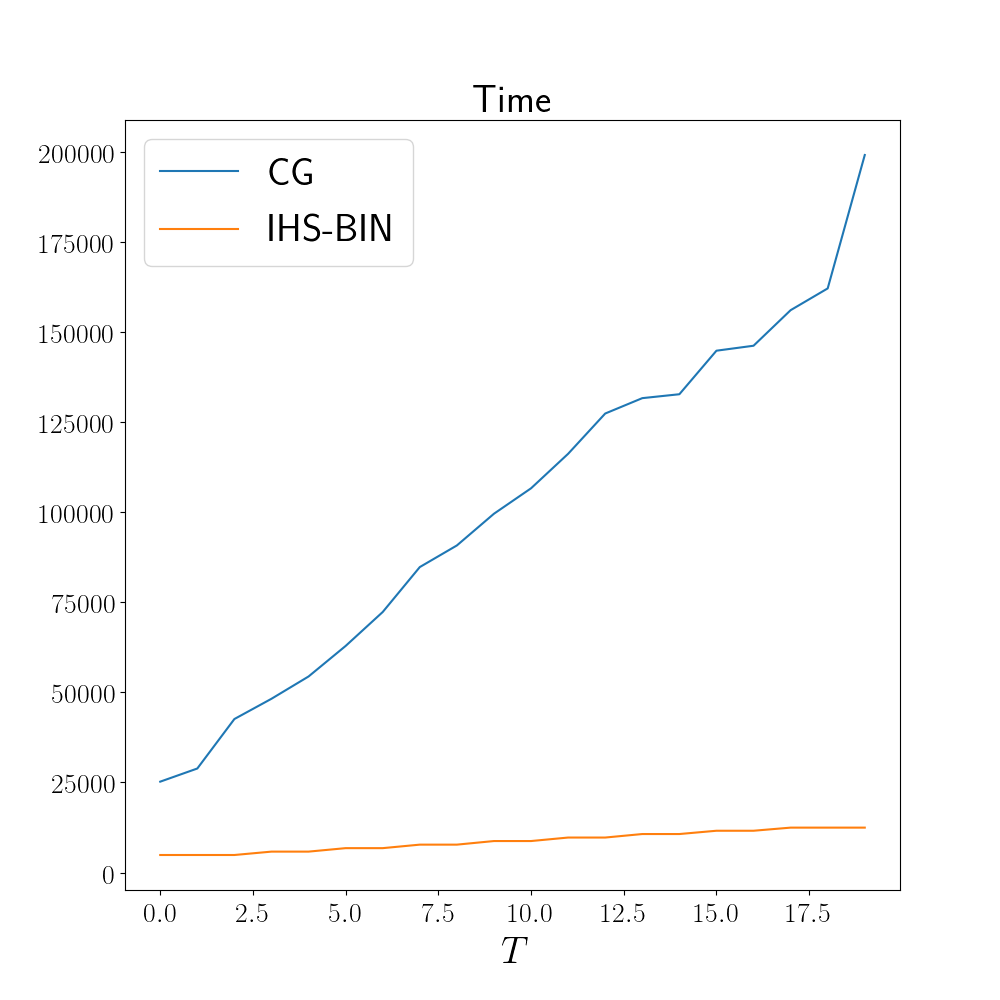}
\end{minipage}
\caption{Training loss, test loss and CPU time on the CIFAR10 dataset. Matrix-valued kernel ridge regression. $n=25000, d=25000, m=10000$. $\lambda_\text{min}=0.1$. $\lambda_\text{max}=10$. We do not calculate the eigenvalues of $A^TA$ since $d$ is large. }
\label{fig:cifar}
\end{figure}
This paper is organized as follows. In section \ref{sec:review}, we review classical methods for solving ridge regression with multiple $\lambda$s. As the motivation of IHS-BIN, we introduce gradient descent with binomial decomposition in section \ref{sec:gd}. We present IHS with binomial decomposition and analyze its convergence rate with different sketching matrices 
in section \ref{sec:ihs}. In section \ref{sec:extension}, we also provide a practical method for estimating an appropriate sketching dimension and extend our algorithm to the under-determined case and matrix-valued ridge regression. 
The numerical results are presented in section \ref{sec:num}.

\section{Review of classical methods}\label{sec:review}
We briefly review several classical methods for solving \eqref{least-square} with $\lambda \in \Lambda$, which contains $T$ distinct values. We focus on the case where $n>d$.

\subsection{Singular value decomposition}
Suppose that the singular value decomposition of $A$ is given by $A=U\Sigma V^T$, where $U\in \mbR^{n\times d},\Sigma\in \mbR^{d\times d}$ and $V\in \mbR^{d\times d}$. 
Then, we can calculate the optimal solution to \eqref{least-square} by 
$$
x^*(\lambda) := (A^TA+\lambda I)^{-1}A^Tb = V(\Sigma^2+\lambda I)^{-1}\Sigma U^Tb
$$
The above expression shows that $x^*(\lambda)$ can be computed for all values of $\lambda$ when the decomposition factors are cached.
The total computation cost of the SVD based computation of the ridge regularization path is therefore $\underbrace{O(nd^2) }_{\textrm{SVD}}+\underbrace{O(d^2T)}_{\mathrm{updating} \,\,\lambda}$. 
\subsection{Warm-started conjugate gradient method}
Suppose that we arrange $\lambda\in \Lambda$ in decreasing order. Then, we can apply the conjugate gradient method to solve \eqref{least-square} with $\lambda\in \Lambda$ from large to small values of $\lambda$. We can use the solution for \eqref{least-square} with larger $\lambda$ as initialization for the next value of $\lambda$. From \citep{hackbusch1994iterative}, the overall computational cost is given by
$$
O(T\text{nnz}(A)\sqrt{\kappa}\log(1/\epsilon)).
$$
Here $\text{nnz}(A)$ denotes the number of non-zero elements in $A$, $\epsilon$ is the tolerance of precision to stop the conjugate gradient method and $\kappa=(\sigma_{\text{max}}(A)+\lambda_\text{min})/(\sigma_{\text{min}}(A)+\lambda_\text{min})$ is the largest condition number of $A+\lambda I$ with $\lambda\in \Lambda$. We let $\sigma_{\text{max}}(A)$ and $\sigma_{\text{min}}(A)$ to represent the largest/smallest singular value of $A$.

\subsection{Warm-started preconditioned conjugate gradient method}
It is well known that the number of iterations of warmed-started CG heavily depends on the condition number $\kappa$. For ill-conditioned data matrices, the condition number $\kappa$ may be very large, which leads to slow convergence of the conjugate gradient method, even using the warm-starts. A widely-applied approach to deal with the large condition number $\kappa$ is to apply a randomized preconditioned conjugate gradient (PCG) method \citep{rokhlin2008fast}. Specifically, we use the random matrix $(A^TS^TSA+\lambda I)^{-1}$ as the preconditioner, where $S\in \mbR^{m\times n}$ is a sketching matrix. The sketching dimension $m$ is usually proportional to the effective dimension $d_e$, which will be discussed in detail in section \ref{ssec:conv}. To calculate the preconditioner $(A^TS^TSA+\lambda I)^{-1}$ for various $\lambda$, we usually compute the SVD of $SA$, which takes $O(d_e^2d)$ time. The computational cost of computing $SA$ can be $O(\log(d_e)\mathrm{nnz}(A))$, depending on the type of the sketch as shown in section \ref{sec:ihs}. Given the preconditioner $(A^TS^TSA+\lambda I)^{-1}$, the computational cost of warm-started PCG is given by
$$
O(T\text{nnz}(A)\log(1/\epsilon)).
$$

\subsection{Warm-started iterative Hessian sketch}
The iterative Hessian sketch (IHS) method \citep{ihs} is a randomized sketching method for solving least square problems. The update rule of IHS follows
\begin{equation}
\begin{aligned}
x_{k+1} = &x_k-\tau(A^TS^TSA+\lambda I)^{-1}A^T(Ax_k-b+\lambda x_k)\\
\end{aligned}
\end{equation}
Here $\tau>0$ is the step size. Similarly, to compute $(A^TS^TSA+\lambda I)^{-1}$ for various $\lambda$, usually we compute the SVD of $SA$, which takes $O(d_e^2d)$ time. With carefully chosen sketching dimension and sketching matrix, IHS can converge in $O(\log(1/\epsilon))$ iterations. Although IHS is simpler, the computational cost of IHS is similar to the PCG method, which is given by
$$
O(T\text{nnz}(A)\log(1/\epsilon)).
$$
We note that the above complexity can be high for large data matrices, especially when $T$, the number of points in the regularization path is also large. In contrast, the proposed approach has complexity $O(Td\log(1/\epsilon))$, which can be significantly smaller when $n$ is large.

\section{Gradient descent regularization path}
\label{sec:gd}
Now we illustrate the main idea underlying our algorithm. Note that gradient descent method with fixed step size for a small number of iterations can approximate the minimizer of \eqref{least-square}. Although this will not be practical, our proposed method is inspired by this approach. Namely, consider the updates
$$
\begin{aligned}
x_{k+1}& = x_k-\tau\pp{A^T(Ax_k-b)+\lambda x_k}\\
&=(I- \tau (A^TA+\lambda I) )x_k+\tau A^Tb\\
&=:Mx_k+\tau A^Tb.
\end{aligned}
$$
Here we denote $M=(I- \tau (A^TA+\lambda I) )$. We can express $x_k$ in terms of $M$ and $x_0$ by
$$
x_k = M^k x_0+\tau M^{k-1}A^Tb+\tau M^{k-2} A^Tb+\dots+\tau A^Tb.
$$
The binomial expansion formula for $M^k$ is given by
$$
M^k=\sum_{j=0}^k{k\choose j}\lambda^{j}(-\tau)^{j}(I- \tau A^TA)^{k-j}.
$$
For simplicity, we assume that the iterations are initialized at $x_0=0$. Then, we can rewrite $x_k$ as
$$
\begin{aligned}
x_k =& \tau\sum_{i=0}^{k-1} M^{i}A^Tb\\
=&\tau \sum_{i=0}^{k-1}\sum_{j=0}^i{i\choose j}\lambda^{j}(-\tau)^{j}(I- \tau A^TA)^{i-j}A^Tb\\
=&\tau \sum_{j=0}^{k-1}(-\tau \lambda)^{j}\sum_{i=0}^{k-1-j} {i+j\choose j}(I- \tau A^TA)^{i}A^Tb\\
=&\tau \sum_{j=0}^{k-1}(-\tau \lambda)^{j}u_j.
\end{aligned}
$$
Here we denote $u_j=\sum_{i=0}^{k-1-j} {i+j\choose j}(I- \tau A^TA)^{i}A^Tb$. Note that the above expression provides an approximate closed form formula for $x(\lambda)$ for all values of $\lambda$. More specifically, if we compute $u_0,\dots,u_{k-1}$, we can instantly compute $x_k=x_k(\lambda)$ for different $\lambda$ parameters. We call this method GD-BIN and summarize it in Algorithm \ref{alg:gd-bin}.

\begin{algorithm2e}[ht]
\caption{Gradient descent with binomial decomposition. (GD-BIN)}\label{alg:gd-bin}
\KwIn{ $A,b,\Lambda=\{\lambda_i\}_{i=1}^T$, iteration number $k$.}

Compute $u_j=\sum_{i=0}^{k-1-j} {i+j\choose j}(I- \tau A^TA)^{i}A^Tb$ for $j=0,\dots,k-1$;
\For{$i=1,\dots,T$}{
Compose $x_i= \tau \sum_{j=0}^{k-1}(-\tau \lambda_i)^{j}u_j$;}

\KwOut{$\{x_i\}_{i=1}^{T}$}
\end{algorithm2e}

However, the convergence rate heavily depends on the condition number $\kappa$ of $A^TA+\lambda_\text{min} I$. 
To obtain an $\epsilon$-approximate solution to \eqref{least-square}, it takes approximately $k=O(\log(1/\epsilon)\kappa)$ iterations. The overall computation cost is as follows
$$
\begin{aligned}
&\underbrace{O(ndk)}_{\text{compute } (I- \tau A^TA)^{i}A^Tb\text{ and }b_j}+\underbrace{O(Tdk)}_{\text{evaluate }x_k}\\
=&\underbrace{O(nd(\log(1/\epsilon)\kappa)}_{\text{compute } (I- \tau A^TA)^{i}A^Tb\text{ and }b_j}+\underbrace{O(Td(\log(1/\epsilon)\kappa)}_{\text{evaluate }x_k}
\end{aligned}
$$

The main drawback of the gradient descent method is the condition number $\kappa$ in the computation cost, which is often prohibitively large in practice. It is interesting to note that the dependence on condition number can be improved to $\sqrt{\kappa}$ using Conjugate Gradient. However, the corresponding regularization path is no longer tractable due to non-constant step-sizes and $\lambda$ can not be updated analogously. To circumvent these difficulties and the dependence on $\kappa$, we take a different approach and introduce the Iterative Hessian sketch (IHS) with binomial decomposition. 

\section{IHS with binomial decomposition}\label{sec:ihs}
Suppose that $S\in \mbR^{m\times n}$ is a sketching matrix, where $m$ is the sketching dimension. IHS \cite{ihs} employs a randomized Newton direction $((\nabla^2 f)^{1/2})^TS^TS(\nabla^2 f)^{1/2})^{-1}$$\nabla f(x)$ to minimize the objective function $f(x)$ in \eqref{least-square}. The sketching dimension $m$ depends on the effective dimension of $A^TA+\lambda I$, which will be defined later and it can be significantly smaller than $d$. Typical choices of sketching matrices include
\begin{itemize}
\item Gaussian sketch: each entry of $S$ follows independent and identically distributed (i.i.d.) Gaussian distribution $\mcN(0,m^{-1})$.
\item CountSketch transform sketch \cite{saatf}: $S$ is initialized as a matrix of zeros. Then, w set $S_{h(i),i}$ to $1$ or $-1$ with equal probability, where $h(i)$ is chosen from $\{1,\dots,n\}$ uniformly at random.
\item Sparse Johnson-Lindenstrauss Transform (SJLT) sketch \cite{sjlt}: with column sparsity $s$, $S$ is constructed by concatenating $s$ independent CountSketch transforms, each of dimension $m/s\times n$.
\item Subsampled Randomized Hadamard Transform (SRHT) sketch \cite{annat}: $S$ is a randomized projection matrix. 
\end{itemize}

The update rule of IHS with a fixed regularization parameter is given by
\begin{equation}\label{equ:ihs}
\begin{aligned}
x_{k+1} = &x_k-\tau(A^TS^TSA+\lambda_0 I)^{-1}A^T(Ax_k-b+\lambda x_k)\\
=&(I- \tau (A^TS^TSA+\lambda_0 I)^{-1}(A^TA+\lambda I) x_k\\
&+\tau (A^TS^TSA+\lambda_0 I)^{-1}A^Tb\\
=&Mx_k+\tau (A^TS^TSA+\lambda_0 I)^{-1} A^Tb.
\end{aligned}
\end{equation}
Here $\lambda_0$ is a fixed parameter. In this case, $M=I- \tau (A^TS^TSA+\lambda_0 I)^{-1}(A^TA+\lambda I) $. 

Based on the update rule of IHS, we can express $x_k$ in terms of $M$ and $x_0$ via
$$
\begin{aligned}
x_k = &M^k x_0+\tau M^{k-1}(A^TS^TSA+\lambda_0 I)^{-1}A^Tb\\
&+\dots+\tau (A^TS^TSA+\lambda_0 I)^{-1} A^Tb.
\end{aligned}
$$
For simplicity, we also assume that the iterations are initialized at $x_0=0$. We denote $P_S = (A^TS^TSA+\lambda_0 I)^{-1} $. 
\begin{proposition}
We can express $x_k$ as a polynomial function of $\lambda$ as follows
$$
x_k=\tau \sum_{j=0}^{k-1}(\tau\lambda)^j\tilde u_j,
$$
Here $\tilde u_j=\sum_{i=j}^{k-1}u_{i,j}$, where $u_{i,j}\in \mbR^d$ can be recursively computed via
\begin{equation}\label{equ:upd}
\begin{aligned}
u_{i+1,0} = (I-\tau P_S A^TA) u_{i,0},\quad u_{i+1,i+1} = -P_S  u_{i,i},\\
u_{i+1,j} = (I-\tau P_S A^TA)u_{i,j}-P_S  u_{i,j-1},\quad 1\leq j\leq i,
\end{aligned}
\end{equation}
with the initial condition $u_{0,0}=P_SA^Tb$. 
\end{proposition}
\begin{proof}
We first claim that for all integer $i\geq 0$,
\begin{equation}\label{equ:dec}
M^iP_SA^Tb = \sum_{j=0}^i (\tau\lambda)^j u_{i,j}.
\end{equation}
We prove this claim by mathematical induction. It is easy to observe that \eqref{equ:dec} hold for $i=0$. Suppose that \eqref{equ:dec} holds for $i$. For $i+1$, we note that
$$
\begin{aligned}
&M^{i+1} A^Tb=M(M^iA^Tb)\\
=& \sum_{j=0}^i \pp{(\tau\lambda)^j (I-P_S A^TA) u_{i,j}-(\tau\lambda)^{j+1} P_S  u_{i,j}}\\
=&\sum_{j=1}^i(\tau\lambda)^j\pp{(I-P_S A^TA) u_{i,j}-P_S  u_{i,j-1}}\\
&+(I-P_S A^TA) u_{i,0}-(\tau\lambda)^{i+1} P_S  u_{i,i}\\
=&\sum_{j=0}^{i+1} (\tau\lambda)^ju_{i+1,j}.
\end{aligned}
$$
Hence, \eqref{equ:dec} also holds for $i+1$. 

As a result, we can easily compute that
$$
\begin{aligned}
x_k = \sum_{i=0}^{k-1}\tau M^{i}P_SA^Tb
= \tau \sum_{i=0}^{k-1} \sum_{j=0}^i (\tau\lambda)^j u_{i,j}
= \tau \sum_{j=0}^{k-1}(\tau\lambda)^j \tilde u_j.
\end{aligned}
$$
\end{proof}

To compute $(A^TS^TSA+\lambda_0 I)^{-1} $, we perform the singular value decomposition on $SA$, i.e., $SA=U_1\Sigma_1V_1$. Suppose that $m<d$. Then, we have
$$
(A^TS^TSA+\lambda_0 I)^{-1} = V_1^T (\Sigma_1^2+\lambda_0 I)^{-1}V_1+\lambda_0^{-1}(I-V_1^TV_1).
$$
Thus, for an arbitrary $v\in \mbR^d$, we have
$$
(A^TS^TSA+\lambda_0 I)^{-1}v = v/\lambda_0+V_1^T( (\Sigma_1^2+\lambda_0 I)^{-1}-\lambda_0^{-1})V_1v.
$$
For $m\geq d$, then, we have
$$
(A^TS^TSA+\lambda_0 I)^{-1}v = V_1^T(\Sigma_1^2+\lambda_0 I)^{-1}V_1v.
$$

We summarize the calculation of $\tilde u_0,\dots,\tilde u_{k-1}$ in Algorithm \ref{alg:basis}.
\begin{algorithm2e}[ht]
\caption{Calculation of basis $\tilde u_0,\dots,\tilde u_{k-1}$ for binomial decomposition.}\label{alg:basis}
\KwIn{ $A,b,P_S,\tau,k$.}
Set $u_i=0$ and $\tilde u_i=0$ with $i=0,\dots,k-1$\;
 Calculate $u_0=-P_S A^Tb$\;
 Let $\tilde u_0=\tilde u_0+u_0$\;
\For{$i=1$ to $k-1$}{
 Calculate $u_i=-P_S u_{i-1}$\;
\For{$j=i-1$ to $1$}{
 Calculate $u_j=u_j-P_S(\tau A^TA u_j+u_{j-1})$\;
}
 Calculate $u_0=u_0-P_S\tau A^TA u_0$\;
 Update $\tilde u_j=\tilde u_j+u_j$ for $j=0,\dots,i$\;
}
\KwOut{$\{\tilde u_j\}_{j=0}^{k-1}$}
\end{algorithm2e}
For computing the sketching $SA$, we list the computation cost for different sketching matrix as follows:
\begin{itemize}
    \item Gaussian sketch: $O(mnd)$ or $O(m\text{nnz}(A))$ for a sparse matrix.
    \item SRHT sketch: $O(\log(m) nd)$ or $O(\log(m)\text{nnz}(A))$ for a sparse matrix.
    \item SJLT sketch: $O(s nd)$ or $O(s\text{nnz}(A))$ for a sparse matrix. Here $s$ is the sparsity of SJLT sketch. 
\end{itemize}
As the sketching dimension is proportional to the effective dimension $d_e$, which can be significantly smaller than $d$, here we assume that $m<d$. Hence, the computation cost to compute the SVD of $SA$ is $O(m^2d)$. Ignoring the complexity of sketching, the computation cost of IHS-BIN follows
$$
\begin{aligned}
\underbrace{O(m^2d)}_{\text{SVD of } SA}+\underbrace{O((md+nd)k^2)}_{\text{compute } \tilde u_j}+\underbrace{O(Tdk)}_{\text{evaluate }x_k}.\\
\end{aligned}
$$
For a sparse matrix $A$, the computation cost reduces to 
$$
\begin{aligned}
&
\underbrace{O(m^2d)}_{\text{SVD of } SA}+\underbrace{O((md+\text{nnz}(A))k^2)}_{\text{compute } \tilde u_j}+\underbrace{O(Tdk)}_{\text{evaluate }x_k}.\\
\end{aligned}
$$
The overall algorithm is summarized in Algorithm \ref{alg:ihs-bin}. 
\begin{algorithm2e}[ht]
\caption{Iterative Hessian Sketch with binomial decomposition. (IHS-BIN)}\label{alg:ihs-bin}
\KwIn{ $A,b,\Lambda=\{\lambda_i\}_{i=1}^T$, iteration number $k$.}
Generate the sketching matrix $S$ and compute the SVD of $SA$;
\For{$i=1,\dots,T$}{
Compute $\{\tilde u_j\}_{j=0}^{k-1}$ using Algorithm \ref{alg:basis};
 Compose $x_i=\tau \sum_{j=0}^{k-1}(\tau\lambda_i)^j\tilde u_j$;
}
\KwOut{ $\{x_i\}_{i=1}^{T}$}
\end{algorithm2e}

\subsection{Sharp estimates of extreme eigenvalues of $C_S$}
We review serveral sharp estimates of $\gamma_1,\gamma_d$ and discuss the probability that $\mcE_S$ holds. For the Gaussian case, we have the following theorem introduced in \cite{edasm}.
\begin{theorem}
Suppose that $S\in \mbR^{m\times n}$ is a Gaussian sketching matrix. Consider $d_e$ defined in \eqref{equ:de} and a parameter $\rho\in(0,1)$. If $m\geq d_e/\rho$, then for any $\eta\in(0,(1-\sqrt{\rho})^2/4)$, with $c(\eta)=\pp{\frac{1+\sqrt{\eta}}{1-\sqrt{\eta}}}^2$ and
$$
\lbb{&\rho_1=1-\|D\|_2^2+\|D\|_2^2(1+\sqrt{\rho})^2(1+\sqrt{\eta})^2,\\
&\rho_2=1-\|D\|_2^2+\|D\|_2^2\pp{1-\sqrt{c(\eta)\rho}}^2,}
$$
the event $\mcE_S$ holds with probability at least $1-16e^{-\eta^2\rho m/2}$
\end{theorem}

For the SJLT sketching, following Lemma 3.3 in \cite{rmihs}, we have an estimate on $\rho_1$ and $\rho_2$.

\begin{theorem}
Suppose that $S\in \mbR^{m\times n}$ is an SJLT sketching matrix with sparsity
\begin{equation}
s=\Omega(\log_\alpha (d_e/\delta)/\epsilon)
\end{equation}
in each column where $\alpha>2, \delta<1/2, \epsilon<1/2$. If the sketch size satisfies
$$
m = \Omega(\alpha d_e\log_\alpha (d_e/\delta)/\epsilon^2),
$$
then, with probability at least $1-\delta$, the event $\mcE_S$ holds. 
\end{theorem}

For the SRHT case, we introduce the relevant factor $$C(m,d_e)=\frac{16}{3}\pp{1+\sqrt{\frac{8\log(d_en)}{d_e}}}^2.
$$
The following theorem in \cite{edasm} provides a sharp estimate on $\rho_1$ and $\rho_2$.
\begin{theorem}
Suppose that $S\in \mbR^{m\times n}$ is an SRHT sketching matrix. Consider $d_e$ defined in \eqref{equ:de} and a parameter $\rho\in(0,1)$. If $m\geq C(n,d_e)\frac{d_e\log(d_e)}{\rho}$. Then, it holds with probability at least $9/d_e$ such that $\mcE_S$ hold with $\rho_1=1+\|D\|_2^2\rho$ and $\rho_2=1-\|D\|_2^2\rho$. 
\end{theorem}

\subsection{Convergence analysis}
\label{ssec:conv}
In this subsection, we analyze the convergence rate of IHS with binomial decomposition. We first introduce some notations. Suppose that $A=U\Sigma V^T$ is the singular value decomposition and let $\sigma_1\geq \dots\geq \sigma_d$ denote singular values of $A$. Let 
$\bar A = \bmbm{A\\\sqrt{\lambda_0}I},\quad \tilde A = \bmbm{A\\\sqrt{\lambda}I}$.
Let $\bar U$ be a matrix of left singular vectors of $\bar A$. For $\lambda_0\geq 0$, we introduce a diagonal matrix $D=\diag\pp{\frac{\sigma_1}{\sqrt{\sigma_1^2+\lambda_0}},\dots,\frac{\sigma_d}{\sqrt{\sigma_d^2+\lambda_0}} }$ and define the effective dimension by
\begin{equation}\label{equ:de}
d_e = \frac{\|D\|_F^2}{\|D\|_2^2}.
\end{equation}
This will influence the sketching dimension, which will be discussed in detail later. Define
$$
\begin{aligned}
\bar \Sigma&=\diag\pp{\sqrt{\sigma_1^2+\lambda_0},\dots,\sqrt{\sigma_d^2+\lambda_0}},\\
 \tilde \Sigma&=\diag\pp{\sqrt{\sigma_1^2+\lambda},\dots,\sqrt{\sigma_d^2+\lambda}}.
\end{aligned}
$$
Denote $\bar S = \bmbm{S&0\\0&I_d}.$ We define two matrices 
$$
\tilde C_S = \tilde \Sigma^{-1} \bar \Sigma C_S \bar \Sigma \tilde \Sigma^{-1},\quad C_S=\bar U^T\bar S^T\bar S\bar U.
$$
Denote $\gamma_1,\gamma_d$ to be the largest/smallest eigenvalue of $C_S$. For two real numbers $\rho_1>\rho_2>0$, we define the $S$-measurable event $\mcE_S=\{\rho_2\leq \gamma_d\leq \gamma_1\leq \rho_1\}.$
We evaluate the error by $\delta_k = \frac{1}{2}\|\tilde A(x_k-x^*)\|^2.$
\begin{theorem}\label{thm:converge}
Suppose that we solve \eqref{least-square} for $\lambda\in[\lambda_\text{min},\lambda_\text{max}]$. 
Denote $\tilde \kappa=\frac{\rho_1\lambda_\text{max}}{\rho_2\lambda_\text{min}}$. By setting $\lambda_0=\sqrt{\lambda_\text{max}\lambda_\text{min}}$ and taking $\alpha=2\sqrt{\rho_1^{-1}\rho_2^{-1}}/(\sqrt{\tilde \kappa}+\sqrt{\tilde \kappa^{-1}})$, then, conditioned on the event $\mcE_S$, the IHS-BIN satisfies that at each iteration
$$
\frac{\delta_{k+1}}{\delta_k}\leq \pp{ \frac{ \tilde \kappa-1}{ \tilde \kappa+1}}^2.
$$
\end{theorem}
\begin{remark}
If we can estimate the smallest singular value $\sigma_d$ of $A$, we can refine the convergence rate as follows. Denote $\hat \kappa=\frac{\rho_1(\lambda_\text{max}+\sigma_d^2)}{\rho_2(\lambda_\text{min}+\sigma_d^2)}$.
Assume that we set $\lambda_0=\sqrt{(\lambda_\text{max}+\sigma_d^2)(\lambda_\text{min}+\sigma_d^2)}-\sigma_d^2$ and take a constant step size $\alpha=2\sqrt{\rho_1^{-1}\rho_2^{-1}}/(\sqrt{\hat \kappa}+\sqrt{\hat \kappa^{-1}})$. Then, conditioned on the event $\mcE_S$, the IHS-BIN satisfies that at each iteration
$$
\frac{\delta_{k+1}}{\delta_k}\leq \pp{ \frac{\hat \kappa-1}{ \hat \kappa+1}}^2.
$$
\end{remark}


We also note that the convergence only depends on the event $\mcE_S$. In other words, as long as the event $\mcE_S$ is satified, we have the convergence guarantee for solving every ridge regression problem along the entire regularization path.

Denote $\beta=\lambda_\text{max}\lambda_\text{min}^{-1}$. Thus, to reach an $\epsilon$ precision solution, it takes $k=O(\log(1/\epsilon)\tilde \kappa)=O(\log(1/\epsilon)\beta)$ iterations. When $\beta$ is large, the iteration number $k$ can be large, which may not be efficient. 

To deal with this problem, we split $[\lambda_\text{min},\lambda_\text{max}]$ into $L$ smaller intervals 
$$
[\lambda^{(0)},\lambda^{(1)}], [\lambda^{(1)},\lambda^{(2)}],\dots, [\lambda^{(L-1)},\lambda^{(L)}].
$$
Here we let $\lambda^{(i)} = \lambda_\text{min}\beta^{i/L}.$ For each small interval $[\lambda^{(i)},\lambda^{(i+1)}]$, it takes approximately $k=\mcO(\log(1/\epsilon)\beta^{1/L})$ iterations to reach an $\epsilon$ precision solution. Compared to the computation cost in computing basis $\tilde u_j$, the computation cost of computing $SA$ and performing SVD of $SA$ is negligble. Hence, the major computation cost follows
$$
\begin{aligned}
&
\underbrace{O(L(md+nd)\log(1/\epsilon)^2\beta^{2/L})}_{\text{compute } \tilde u_j \; L\text{ times}}+\underbrace{O(Td\log(1/\epsilon)\beta^{1/L})}_{\text{evaluate }x_k}.\\
\end{aligned}
$$
Suppose that we take $L=\lfloor 2 \log \beta \rfloor$. Then, the computation cost writes
$$
\begin{aligned}
&
\underbrace{O((md+nd)\log \beta  \log(1/\epsilon)^2)}_{\text{compute } \tilde u_j \; L\text{ times}}+\underbrace{O(Td\log(1/\epsilon))}_{\text{evaluate }x_k}.\\
\end{aligned}
$$
Similarly, if $A$ is a sparse matrix, then the computation cost follows
$$
\begin{aligned}
&
\underbrace{O((md+\text{nnz}(A))\log \beta  \log(1/\epsilon)^2)}_{\text{compute } \tilde u_j \; L\text{ times}}+\underbrace{O(Td\log(1/\epsilon))}_{\text{evaluate }x_k}.\\
\end{aligned}
$$
Namely, we can improve the dependence of $\beta$ in computing binomial basis $\tilde u_j$ from $\beta^2$ to $\log \beta$. 

\subsection{Proof of Theorem \ref{thm:converge}}
For the update rule, we have
$$
x_{k+1} = x_k-\alpha\pp{\bar A^T\bar S^T\bar S\bar A}^{-1}\tilde A^T(\tilde A x_k-\bar b).
$$
We use the notation $e_t=\tilde U^T\tilde A(x_k-x^*)$. Here $x^*$ is the unique minimizer of \begin{equation}
\min_{x\in \mbR^d} \frac{1}{2}\|Ax-b\|_2^2+\frac{\lambda}{2}\|x\|_2^2.
\end{equation}
We obtain that
$$
\begin{aligned}
e_{k+1} =& e_k-\alpha \tilde U^T \tilde A \pp{\bar A^T\bar S^T\bar S\bar A}^{-1}\tilde A^T\tilde U e_k\\
=&\pp{I-\alpha \tilde \Sigma (\bar \Sigma\bar U^T\bar S^T\bar S\bar U\bar \Sigma )^{-1}\tilde \Sigma}e_k\\
=&(I-\alpha \tilde C_S^{-1})e_k.
\end{aligned}
$$
We note that
$$
\begin{aligned}
\sqrt{\delta_{k+1}}=&\|e_{k+1}\|_2\leq \|I-\alpha \tilde C_S^{-1}\|_2 \|e_k\|_2\\
=&\|I-\alpha \tilde C_S^{-1}\|_2 \sqrt{\delta_k}.
\end{aligned}
$$
Then, the convergence rate depend on the condition number of $\tilde C_S$. For the rest of the proof, we assume that the event $\mcE_S$ holds. Based on the estimations $\rho_1,\rho_2$ for the extreme eigenvalues of $C_S$, we define $\tilde \rho_1(\lambda)>\tilde \rho_2(\lambda)>0$ as follows:
if $\lambda\geq \lambda_0$, we let
$$
\tilde \rho_2(\lambda) = \frac{\sigma_d^2+\lambda_0}{\sigma_d^2+\lambda}\rho_2,\quad \tilde \rho_1(\lambda) = \frac{\sigma_1^2+\lambda_0}{\sigma_1^2+\lambda}\rho_1.
$$
Otherwise, if $\lambda\leq \lambda_0$, we let
$$
\tilde \rho_2(\lambda) = \frac{\sigma_1^2+\lambda_0}{\sigma_1^2+\lambda}\rho_2,\quad \tilde \rho_1(\lambda) = \frac{\sigma_d^2+\lambda_0}{\sigma_d^2+\lambda}\rho_1.
$$ 
Hence, the extreme eigenvalues of $\tilde C_S$ is bounded in $[\tilde \rho_2(\lambda), \tilde \rho_1(\lambda)]$. Hence, the convergence rate follows
$$
\sqrt{\frac{\delta_{k+1}}{\delta_k}}\leq\max\{\vert1-\tilde \alpha \rho_1(\lambda)^{-1}\vert,\vert 1-\alpha \tilde \rho_2(\lambda)^{-1}\vert\}.
$$
For $\lambda\in[\lambda_{\text{min}},\lambda_{\text{max}}]$, we want to minimize the worst convergence rate:
$$
\min_{\alpha>0,\lambda_0\in [\lambda_\text{min},\lambda_{\text{max}}]}\max_{\lambda\in[\lambda_\text{min},\lambda_{\text{max}}]}\max\{\vert1- \alpha\tilde \rho_1(\lambda)^{-1}\vert,\vert1-\alpha \tilde \rho_2(\lambda)^{-1}\vert\}.
$$
For $\lambda\geq \lambda_0$, we have 
$$
\tilde \rho_2(\lambda) = \frac{\sigma_d^2+\lambda_0}{\sigma_d^2+\lambda}\rho_2,\quad \tilde \rho_1(\lambda) = \frac{\sigma_1^2+\lambda_0}{\sigma_1^2+\lambda}\rho_1,
$$
which yields
$$
\tilde \rho_2(\lambda_\text{max})\leq \tilde \rho_2(\lambda)\leq \tilde \rho_1(\lambda)\leq \tilde \rho_1(\lambda_\text{max}).
$$
This indicates that
$$
\begin{aligned}
&\max_{\lambda\in[ \lambda_0,\lambda_\text{max}]}\max\{\vert1- \alpha\tilde \rho_1(\lambda)^{-1}\vert,\vert1-\alpha \tilde \rho_2(\lambda)^{-1}\vert\}\\
=&\max\{\vert1- \alpha\tilde \rho_1(\lambda_\text{max})^{-1}\vert,\vert1-\alpha \tilde \rho_2(\lambda_\text{max})^{-1}\vert\}.
\end{aligned}
$$
For $\lambda\leq \lambda_0$, similarly, we have
$$
\tilde \rho_2(\lambda) = \frac{\sigma_1^2+\lambda_0}{\sigma_1^2+\lambda}\rho_2,\quad \tilde \rho_1(\lambda) = \frac{\sigma_d^2+\lambda_0}{\sigma_d^2+\lambda}\rho_1.
$$
This indicates that
$$
\begin{aligned}
&\max_{\lambda\in[ \lambda_\text{min},\lambda_0]}\max\{\vert1- \alpha\tilde \rho_1(\lambda)^{-1}\vert,\vert1-\alpha \tilde \rho_2(\lambda)^{-1}\vert\}\\
=&\max\{\vert1- \alpha\tilde \rho_1(\lambda_\text{min})^{-1}\vert,\vert1-\alpha \tilde \rho_2(\lambda_\text{min})^{-1}\vert\}.
\end{aligned}
$$
We further notice that
$$
\tilde \rho_2(\lambda_\text{max})\leq \rho_2\leq \tilde \rho_2(\lambda_\text{min}),\quad \tilde \rho_1(\lambda_\text{max})\leq \rho_1\leq \tilde \rho_1(\lambda_\text{min}). 
$$
In short, we have
$$
\begin{aligned}
&\max_{\lambda\in[ \lambda_\text{min},\lambda_\text{max}]} \max\{\vert1- \alpha\tilde \rho_1(\lambda)^{-1}\vert,\vert1-\alpha \tilde \rho_2(\lambda)^{-1}\vert\}\\
=&\max\left\{\left\vert1- \tilde \rho_2(\lambda_\text{max})^{-1}\alpha \right\vert ,\left\vert1-\tilde \rho_1(\lambda_\text{min})^{-1}\alpha \right\vert\right\}\\
=&\max\left\{\left\vert1- \frac{\sigma_d^2+\lambda_\text{max}}{\sigma_d^2+\lambda_0}\rho_2^{-1}\alpha \right\vert ,\left\vert1- \frac{\sigma_d^2+\lambda_\text{min}}{\sigma_d^2+\lambda_0}\rho_1^{-1}\alpha \right\vert\right\}.
\end{aligned}
$$
The above quantity is minimized when 
$$
\alpha/(\lambda_0+\sigma_d^2)=2\pp{(\sigma_d^2+\lambda_\text{min})\rho_1^{-1}+(\sigma_d^2+\lambda_\text{max})\rho_2^{-1}}^{-1}.
$$
Thus, by taking $\lambda_0=\sqrt{(\lambda_\text{max}+\sigma_d^2)(\lambda_\text{min}+\sigma_d^2)}-\sigma_d^2$, the optimal step size follows $\alpha=2\sqrt{\rho_1^{-1}\rho_2^{-1}}/\pp{\sqrt{\kappa}+\sqrt{\kappa^{-1}}}$. In summary, we have
$$
\begin{aligned}
\sqrt{\frac{\delta_{k+1}}{\delta_k}}\leq&
\min_{\alpha>0,\lambda_0\in [\lambda_\text{min},\lambda_{\text{max}}]}\max_{\lambda\in[\lambda_\text{min},\lambda_{\text{max}}]}\\
&\max\{\vert1- \alpha\tilde \rho_1(\lambda)^{-1}\vert,\vert1-\alpha \tilde \rho_2(\lambda)^{-1}\vert\}\\
=& \frac{\kappa-1}{\kappa+1}.
\end{aligned}
$$
For the case where $\sigma_d$ is unknown, since 
$$
\frac{\lambda_\text{max}}{\lambda_0}\rho_2^{-1}\geq \frac{\sigma_d^2+\lambda_\text{max}}{\sigma_d^2+\lambda_0}\rho_2^{-1}\geq \frac{\sigma_d^2+\lambda_\text{min}}{\sigma_d^2+\lambda_0}\rho_1^{-1}\geq \frac{\lambda_\text{min}}{\lambda_0}\rho_1^{-1},
$$
we can relax the bound as follows
$$
\begin{aligned}
&\max\left\{\left\vert1- \frac{\sigma_d^2+\lambda_\text{max}}{\sigma_d^2+\lambda_0}\rho_2^{-1}\alpha \right\vert ,\left\vert1- \frac{\sigma_d^2+\lambda_\text{min}}{\sigma_d^2+\lambda_0}\rho_1^{-1}\alpha \right\vert\right\}\\
\leq &\max\left\{\left\vert1- \frac{\lambda_\text{max}}{\lambda_0}\rho_2^{-1}\alpha \right\vert ,\left\vert1- \frac{\lambda_\text{min}}{\lambda_0}\rho_1^{-1}\alpha \right\vert\right\}.
\end{aligned}
$$
Similarly, the above quantity is minimized when $\alpha/(\lambda_0)=2\pp{\lambda_\text{min}\rho_1^{-1}+\lambda_\text{max}\rho_2^{-1}}^{-1}$. For $\lambda_0=\sqrt{\lambda_\text{min}\lambda_\text{max}}$, the optimal step size writes $\alpha=2\sqrt{\rho_1^{-1}\rho_2^{-1}}/\pp{\sqrt{\tilde \kappa}+\sqrt{\tilde \kappa^{-1}}}$. We then have
$$
\begin{aligned}
&\sqrt{\frac{\delta_{k+1}}{\delta_k}}\leq \min_{\alpha>0,\lambda_0\in [\lambda_\text{min},\lambda_{\text{max}}]}\max_{\lambda\in[\lambda_\text{min},\lambda_{\text{max}}]}\\
&\max\{\vert1- \alpha\tilde \rho_1(\lambda)^{-1}\vert,\vert1-\alpha \tilde \rho_2(\lambda)^{-1}\vert\}\\
=& \frac{\tilde \kappa-1}{\tilde \kappa+1}.
\end{aligned}
$$

\section{Estimation of the effective dimension and extensions of IHS-BIN}\label{sec:extension}

\subsection{Estimation of the effective dimension}
From previous theorems, an appropriate sketching size depends on the effective dimension $d_e$. Nevertheless, in practice, usually the estimation of $d_e$ is available when $d_e$ is small, see \cite{sbfrd}. Following the adaptive method described in \cite{edasm}, we propose a practical method for finding an appropriate sketching size. 

We apply IHS to solve \eqref{least-square} with $\lambda=\lambda_\text{min}$. In $k$-th iteration, we first calculate the direction
$$
d_k = (A^TS^TSA+\lambda_\text{min} I)^{-1}A^T(Ax_k-b+\lambda_\text{min} x_k).
$$
Then, we calculate a step size $\tau_k=\gamma_1^{j}$ satisfying the Armijo condition. Namely, $j$ is the smallest non-negative integer satisfying
\begin{equation}\label{armijo}
    \begin{aligned}
&f(x_k-\gamma_1^jd_k)\leq f(x_k)-\gamma_2\gamma_1^jd_k^T\nabla f(x_k).
\end{aligned}
\end{equation}
Here we write $f(x)=\frac{1}{2}\|Ax-b\|_2^2+\frac{\lambda_\text{min}}{2}\|x\|_2^2$ and $\gamma_1,\gamma_2\in (0,1)$ are parameters. Then, we update $x_{k+1}=x_k-\tau_kd_k$. 

We evaluate the following quantity per iteration
$$
\tilde \delta_k = d_k^TA^T(Ax_k-b+\lambda_\text{min} x_k).
$$
If $\tilde \delta_{k+1}\geq \gamma_3 \delta_k$ for some $\gamma_3>0$, then, we let $m=2m$ and sample the sketching matrix. We stop the algorithm when $\tilde \delta_k<\epsilon$ for some $\epsilon>0$. The whole algorithm is described in Algorithm \ref{alg:sketch}.

\begin{algorithm2e}[ht]
\caption{Adaptive estimation of sketching dimension}\label{alg:sketch}
\KwIn{ $A,b,x_0,\lambda_\text{min}$, $\epsilon$, $\gamma_1, \gamma_2, \gamma_3$.}
 Set $k=0$. Compute $d_0$ and $\tilde \delta_0$\;
\While{$\tilde \delta_k>\epsilon$}
{Calculate a step size $\tau_k=\gamma_1^{j}$ satisfying the Armijo condition \eqref{armijo}\;
Update $x_{k+1}=x_k-\tau_kd_k$\;
Calculate $d_{k+1}$ and $\tilde \delta_{k+1}$\;}
\eIf{$\tilde \delta_{k+1}\geq \gamma_3 \delta_k$}
{Set $m=2m$ and sample the sketching matrix $S$\;
 Recompute $d_k$ based on $S$\;}
{Set $k=k+1$;}
\KwOut{$m$}
\end{algorithm2e}

\subsection{Extension to the under-determined case}
For the under-determined case where $n<d$, we consider the dual problem of \eqref{least-square}:
\begin{equation}\label{ls_reg:d}
\min_{z\in \mbR^m} \frac{1}{2}\|A^Tz\|^2+\frac{\lambda}{2}\|z\|^2-b^Tz.
\end{equation}
The optimal solution $z$ to the dual problem is related to the optimal solution to the primal problem by
$$
v=A^Tz.
$$
Hence, we can apply similar methods in the under-parameterized case to solve \eqref{ls_reg:d}.

\subsection{Extension for matrix-valued ridge-regression}

Consider the following matrix-valued ridge-regression problem:
\begin{equation}\label{least-square-max}
\min_{x\in \mbR^{d\times K}} \frac{1}{2}\|AX-B\|_F^2+\frac{\lambda}{2}\|X\|_F^2,
\end{equation}
where $A\in \mbR^{n\times d}$ and $B\in \mbR^{n\times K}$. We can easily extend IHS-BIN for solving this problem.

\begin{algorithm2e}[ht]
\caption{Calculation of basis $\tilde u_0,\dots,\tilde u_{k-1}$ for IHS-BIN with matrix-valued ridge regression.}\label{alg:basis_mat}
\KwIn{$A,B,P_S,\tau,k$.}
 Set $U_i=0$ and $\tilde U_i=0$ with $i=0,\dots,k-1$\;
 Calculate $U_0=-P_S A^TB$\;
 Let $\tilde U_0=\tilde U_0+U_0$\;
\For{$i=1$ to $k-1$}{
 Calculate $U_i=-P_S U_{i-1}$\;
\For{$j=i-1$ to $1$}{
 Calculate $U_j=U_j-P_S(\tau A^TA U_j+U_{j-1})$\;
}
 Calculate $U_0=U_0-P_S\tau A^TA U_0$\;
 Update $\tilde U_j=\tilde U_j+U_j$ for $j=0,\dots,i$\;
}
\KwOut{ $\{\tilde U_j\}_{j=0}^{k-1}$}
\end{algorithm2e}

\begin{algorithm2e}[ht]
\caption{IHS-BIN with matrix-valued ridge regression.}\label{alg:ihs-bin-mat}
\KwIn{ $A,b,\Lambda=\{\lambda_i\}_{i=1}^T$, iteration number $k$.}
Generate the sketching matrix $S$ and compute the SVD of $SA$\;
\For{$i=1,\dots,T$}{
     Compute $\{\tilde U_j\}_{j=0}^{k-1}$ using Algorithm \ref{alg:basis_mat}\;
     Compose $X_i=\tau \sum_{j=0}^{k-1}(\tau\lambda_i)^j\tilde U_j$\;
}
\KwOut{$\{X_i\}_{i=1}^T$}
\end{algorithm2e}

Neglecting the computation cost of computing $SA$, the computation cost of IHS-BIN becomes
$$
\begin{aligned}
&
\underbrace{O(m^2d)}_{\text{SVD of } SA}+\underbrace{O(K(md+nd)\log \beta  \log(1/\epsilon)^2)}_{\text{compute } \tilde u_j \; L\text{ times}}+\underbrace{O(TKd\log(1/\epsilon))}_{\text{evaluate }x_k},
\end{aligned}
$$
for a dense matrix $A$, or
$$
\begin{aligned}
&
\underbrace{O(m^2d)}_{\text{SVD of } SA}+\underbrace{O(K(md+\text{nnz}(A))\log \beta  \log(1/\epsilon)^2)}_{\text{compute } \tilde u_j \; L\text{ times}}
+\underbrace{O(TKd\log(1/\epsilon))}_{\text{evaluate }x_k},
\end{aligned}
$$
for a sparse matrix $A$. For comparison, the computation cost of SVD-based method follows
$$
\underbrace{O(n^2d) }_{\text{SVD}}+O(d^2KT).
$$
The computation cost of CG-based method writes
$$
O(TKnd\sqrt{\kappa}\log(1/\epsilon) )\text{  or  }O(TK\text{nnz}(A)\sqrt{\kappa}\log(1/\epsilon) ).
$$

\section{Numerical results}\label{sec:num}
In this section, we present numerical comparisons between IHS-BIN and other methods for solving least-square problems with various regularization parameters. We test over randomly generated data and real data. For real data, we collect datasets from LIBSVM\footnote{https://www.csie.ntu.edu.tw/~cjlin/libsvm/} under the modified BSD license. We randomly split half of the data matrix as the training data $A$ and the other as the test data matrix $\tilde A$. We denote $b$ and $\tilde b$ as the corresponding labels of $A$ and $\tilde A$. All numerical experiments are conducted on a Dell PowerEdge R840 workstation (64 core, 3TB ram). Detailed setups for the datasets and algorithms are given in the appendix. We provide numerical comparisons with the following baseline algorithms. \textbf{SVD:} the Singular Value Decomposition based method., \textbf{native:} the native linear system solver in NumPy,  \textbf{CG:} warm-started Conjugate Gradient method. For IHS-BIN, we use the SJLT sketching matrices with sparsity $1$. The code can be found in \url{https://github.com/pilancilab/IHS-BIN}.

\subsection{Datasets setup in the numerical experiments}
For the randomly generated data, each row of $A$ follows $\mcN(0,\Sigma^2/\sqrt{nd})$, where $\Sigma_{i,j}=\alpha^{|i-j|}$. Here we let $\alpha=0.99$. 
The training loss and the test loss follows
$$
L_\text{train}=\frac{1}{2}\|Av-b\|^2_2+\frac{\lambda}{2}\|v\|^2,\quad L_\text{test}=\frac{1}{2}\|\tilde Av-\tilde b\|^2_2.
$$
Here each row of $\tilde A$ follows $\mcN(0,\Sigma^2/(nd))$. And we let
$$
b=Av^*+\eta,\quad \tilde b=\tilde Av^*+\tilde \eta.
$$
Here $v^*\sim \mcN(0,I_d/d)$ and $\eta,\tilde \eta\sim \mcN(0,\sigma^2)$ where $\sigma>0$ is a parameter. 

For real data, we linearly rescale the entry of $A$ into $[-1,1]$. Besides, for CIFAR10, we use the kernel matrix in ridge regression. Namely, $A$ is the kernel matrix formulated by the training data and $\tilde A$ is the kernel matrix formulated by the training data and test data. In short, we have
$$
A_{i,j} = k(f_i,f_j), \tilde A_{i,j}=k(\tilde f_i, f_j),
$$
where $k(x,y):\mbR^d\times\mbR^d\to\mbR$ is a positive definite kernel function, $f_i$ is the feature vector of $i$-th training sample and $\tilde f_j$ is the feature vector of $j$-th test sample. Here we use an isotropic Gaussian kernel function
$$
k(x,y)=(2\pi h)^{-d/2} \exp(-\frac{1}{2h}\|x-y\|_2^2),
$$
with bandwidth $h=1000$.

For IHS-BIN, we split the interval $[\lambda_\text{min},\lambda_\text{max}]$ into $L=\lfloor 2\log(\lambda_\text{max}/\lambda_\text{min})\rfloor$ small intervals.
\subsection{Over-determined case}

We present numerical results on randomly generated data and real data in Figure \ref{fig:random_od}, \ref{fig:realsim} and \ref{fig:avazu}. In Figure \ref{fig:cifar}, we present results on matrix-valued ridge regression with kernel matrix. For the problem with medium-size $d$ (randomly generated data), we plot the eigenvalues of $A^TA$. The curves of train loss and test loss from IHS-BIN overlap with curves from other solvers. This indicates that IHS-BIN yields correct solutions along the regularization path. For medium-scale problems like randomly generated data and real-sim, CG and IHS-BIN outperform SVD and native. For the large-scale problem avazu and matrix-valued ridge regression, IHS-BIN can be significantly faster than CG.

\begin{figure}[!htbp]
\centering
\begin{minipage}[t]{\figsize\textwidth}
\centering
\includegraphics[width=\linewidth]{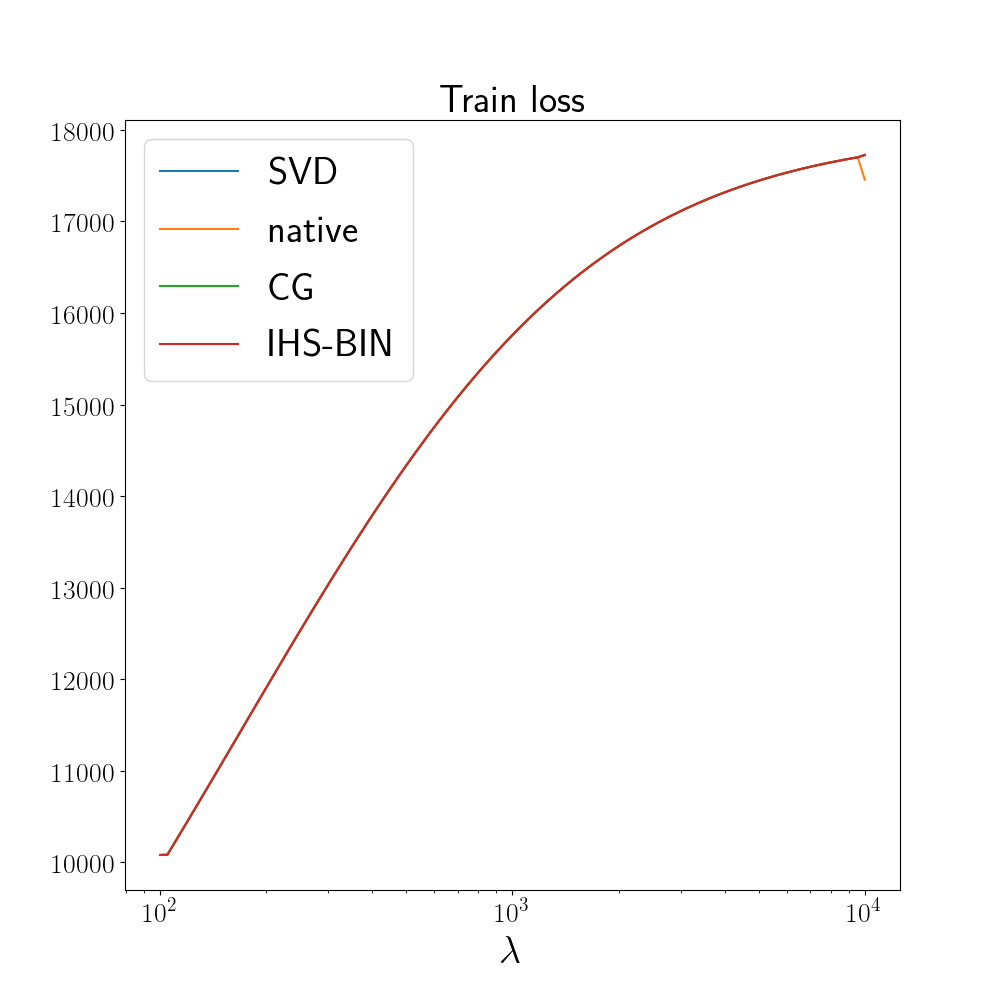}
\end{minipage}
\begin{minipage}[t]{\figsize\textwidth}
\centering
\includegraphics[width=\linewidth]{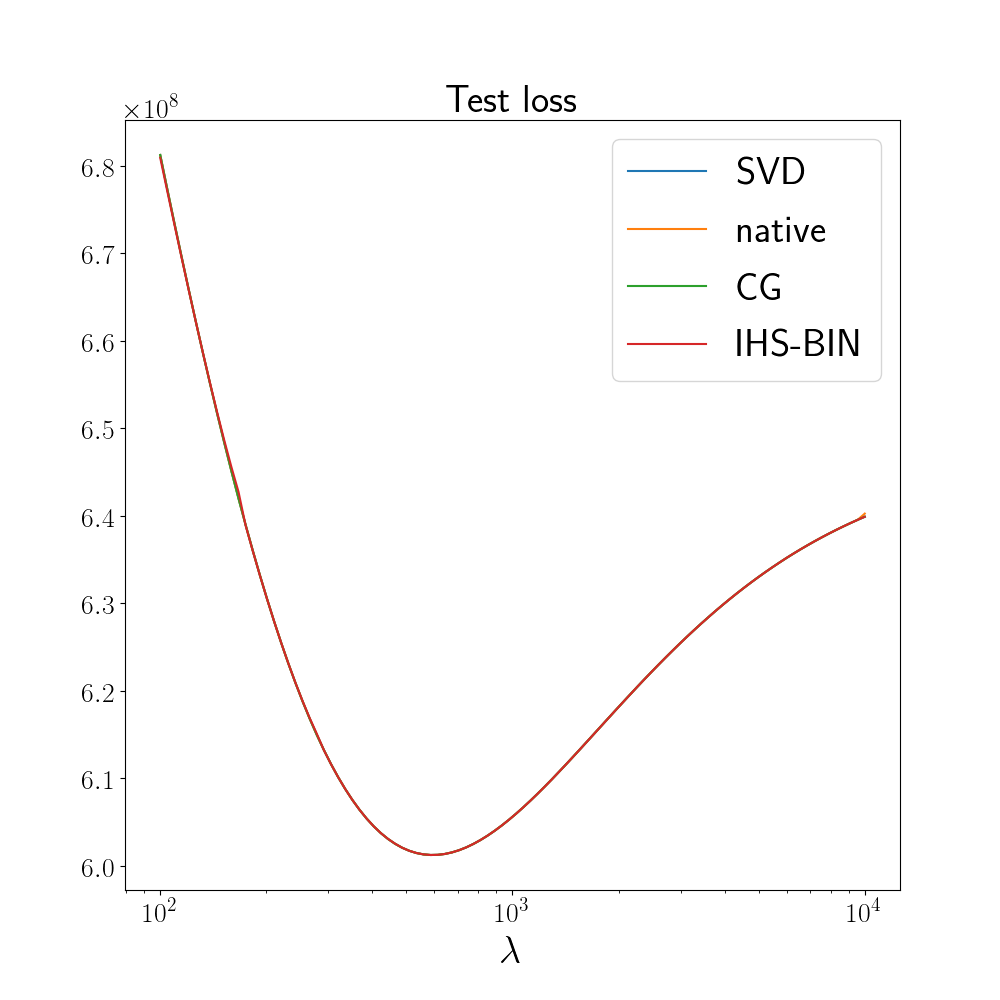}
\end{minipage}
\begin{minipage}[t]{\figsize\textwidth}
\centering
\includegraphics[width=\linewidth]{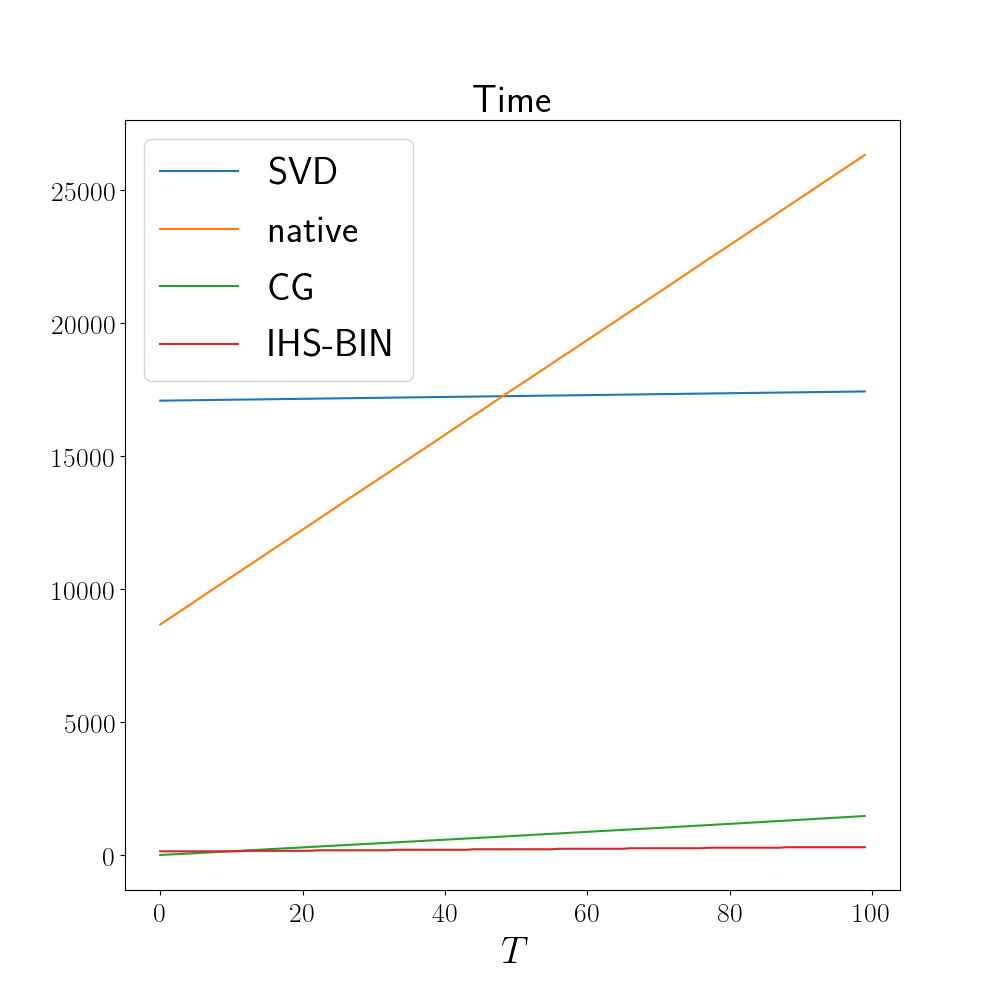}
\end{minipage}
\caption{Training loss, test loss and time. Real vs. Simulated (real-sim). $n=36000, d=20958, m=8000$. $\lambda_\text{min}=100$. $\lambda_\text{max}=10^4$. We do not calculate the eigenvalues of $A^TA$ since $d$ is large. }
\label{fig:realsim}
\end{figure}

\begin{figure}[!htbp]
\centering
\begin{minipage}[t]{\figsize\textwidth}
\centering
\includegraphics[width=\linewidth]{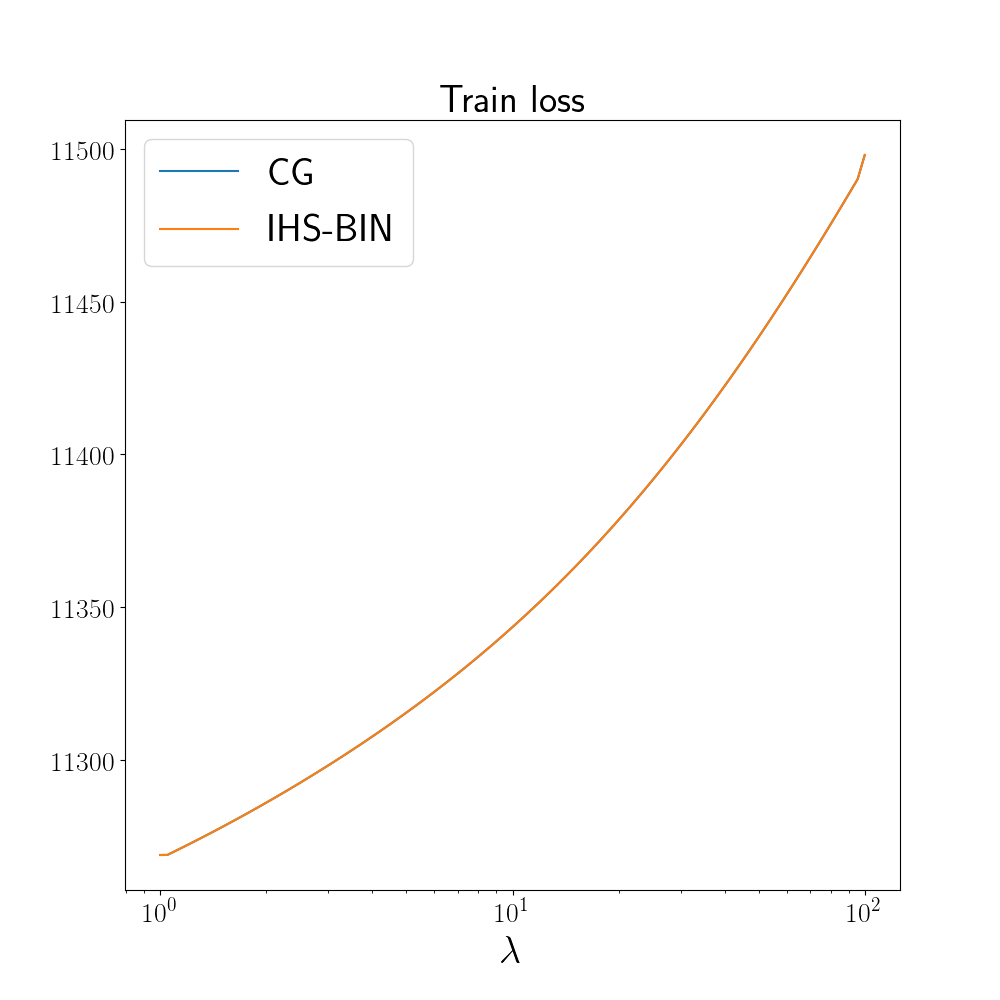}
\end{minipage}
\begin{minipage}[t]{\figsize\textwidth}
\centering
\includegraphics[width=\linewidth]{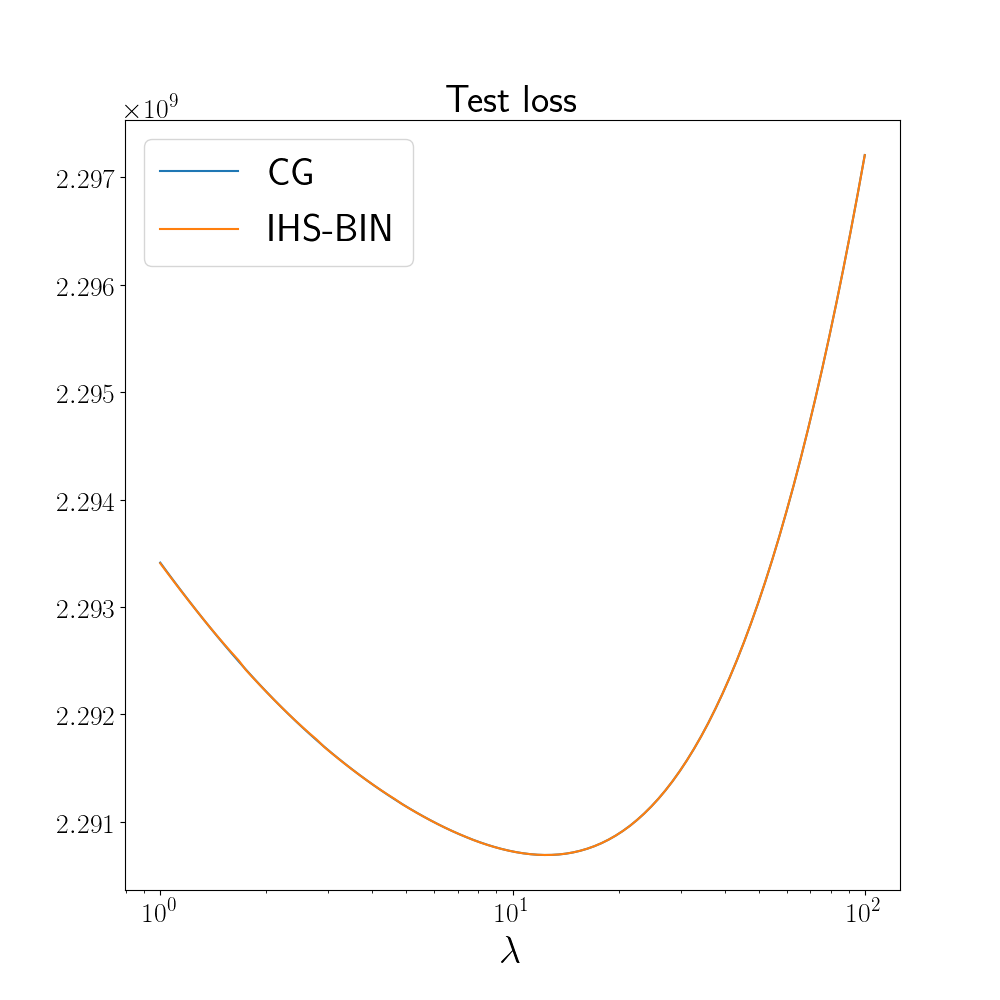}
\end{minipage}
\begin{minipage}[t]{\figsize\textwidth}
\centering
\includegraphics[width=\linewidth]{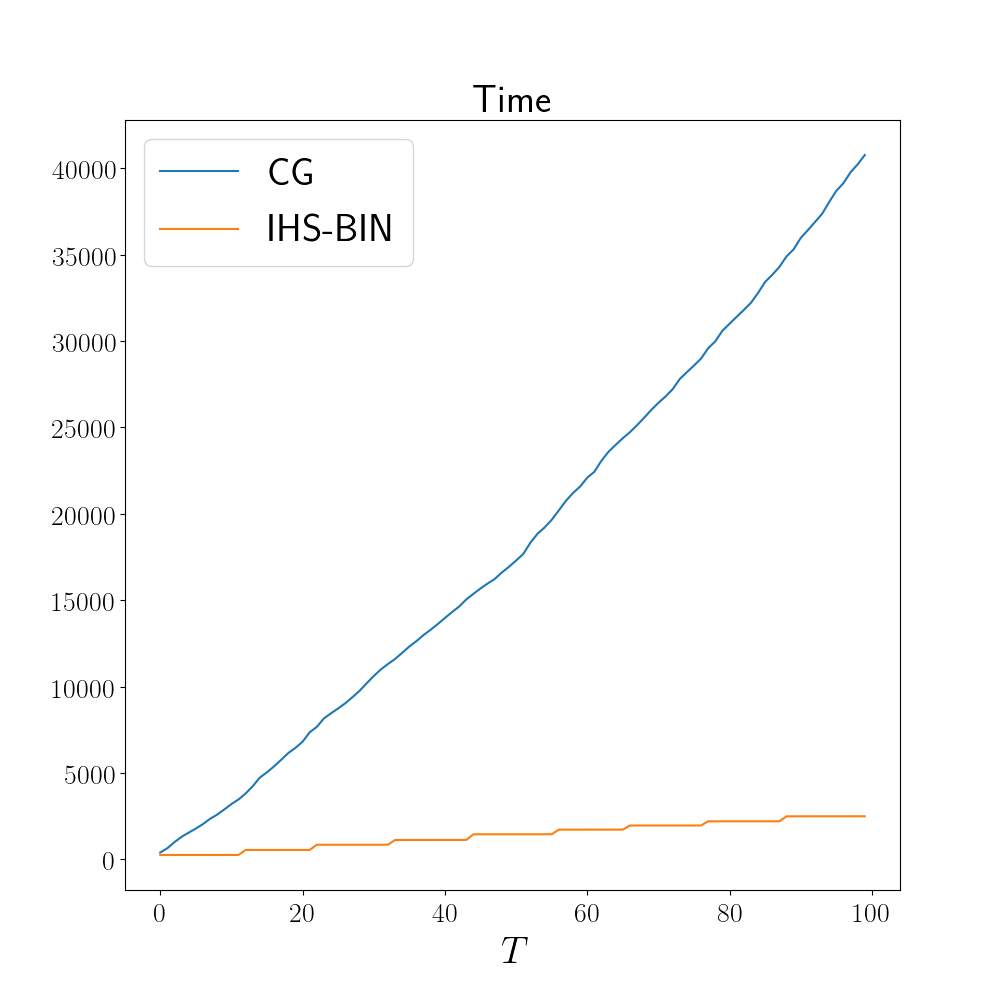}
\end{minipage}
\caption{Training loss, test loss and time. Avazu's Click-through Prediction (avazu). $n=200000, d=50000, m=10000$. $\lambda_\text{min}=1$. $\lambda_\text{max}=100$.  We do not calculate the eigenvalues of $A^TA$ since $d$ is large. }
\label{fig:avazu}
\end{figure}

\subsection{Under-determined case}
We also perform numerical comparisons on under-determined data matrices. We present numerical results in Figure \ref{fig:mnist-kron} to Figure \ref{fig:tfidf}. For the problems with medium-size $n$ (randomly generated data, gisette, RCV1, tifdf), we plot the eigenvalues of $AA^T$.  Similar to the over-determined case, IHS-BIN is faster than other compared methods, especially for large0scale examples. The curves of train loss and test loss from IHS-BIN overlap with curves from other solvers, which implies the accuracy of IHS-BIN.
\begin{figure}[!htbp]
\centering
\begin{minipage}[t]{\figsize\textwidth}
\centering
\includegraphics[width=\linewidth]{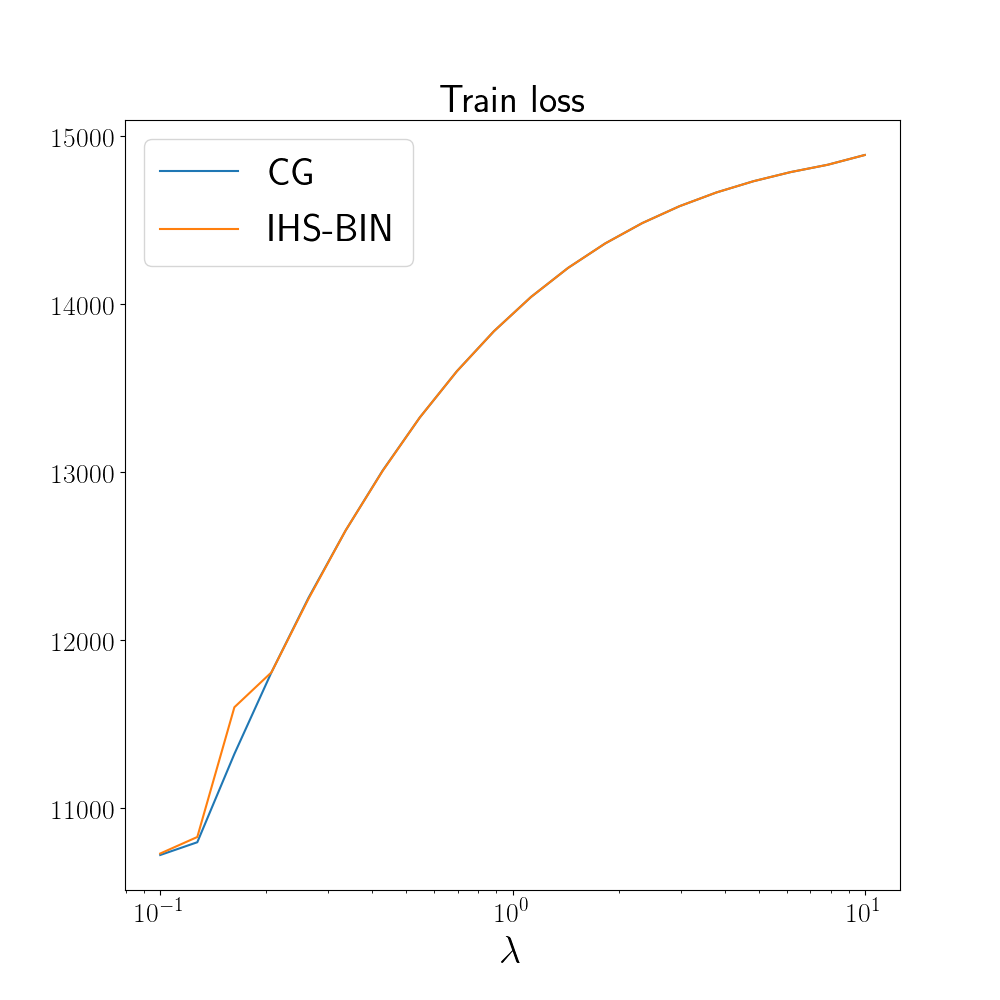}
\end{minipage}
\begin{minipage}[t]{\figsize\textwidth}
\centering
\includegraphics[width=\linewidth]{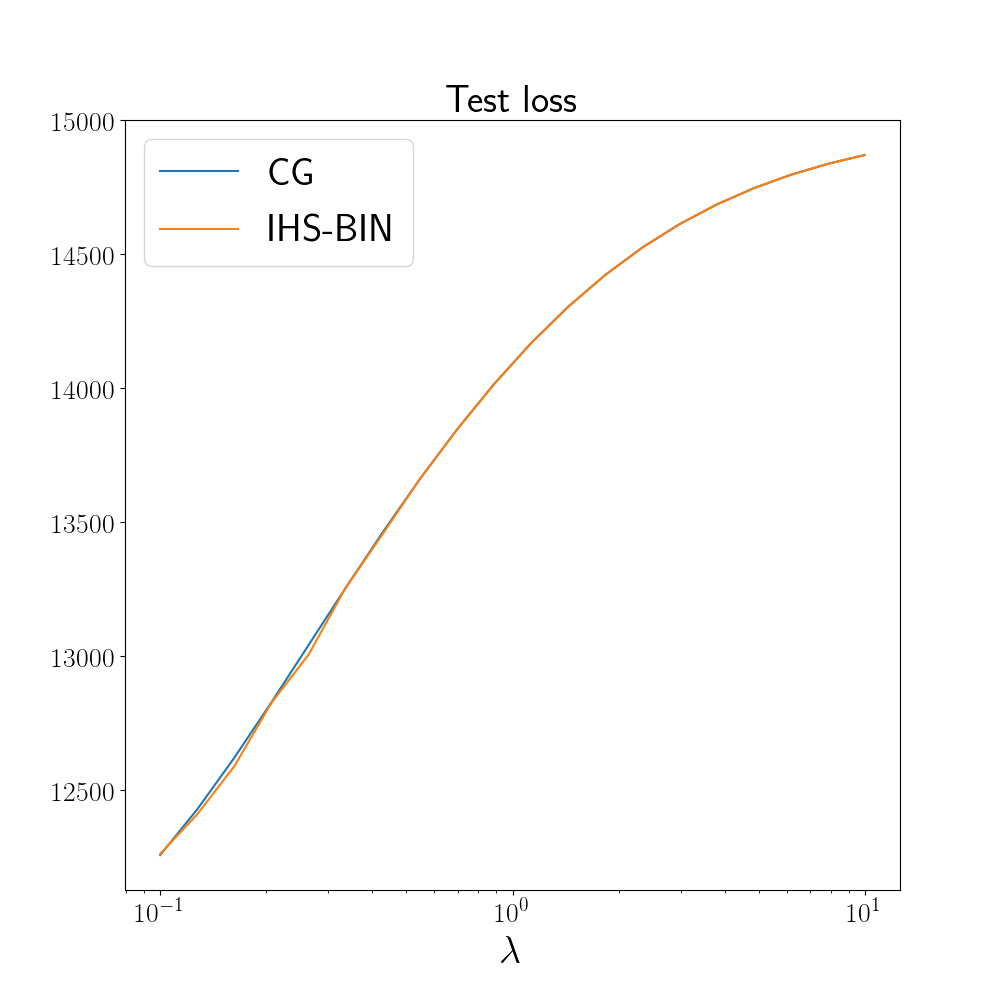}
\end{minipage}
\begin{minipage}[t]{\figsize\textwidth}
\centering
\includegraphics[width=\linewidth]{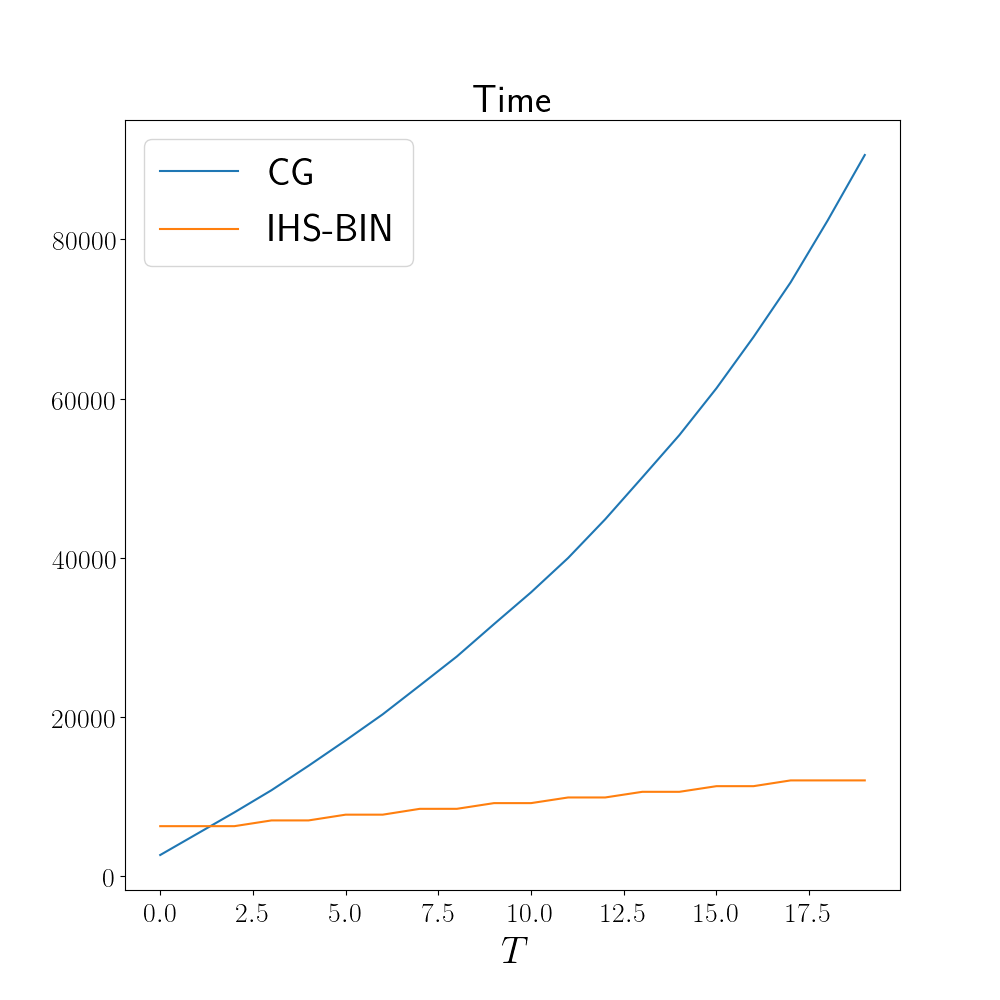}
\end{minipage}
\caption{Training loss, test loss and time. MNIST with quadratic feature embedding. $n=30000, d=608400, m=10000$. $\lambda_\text{min}=0.1$. $\lambda_\text{max}=10$. We do not calculate the eigenvalues of $AA^T$ since $n$ is large. }
\label{fig:mnist-kron}
\end{figure}

\begin{figure}[!htbp]
\centering
\begin{minipage}[t]{\figsize\textwidth}
\centering
\includegraphics[width=\linewidth]{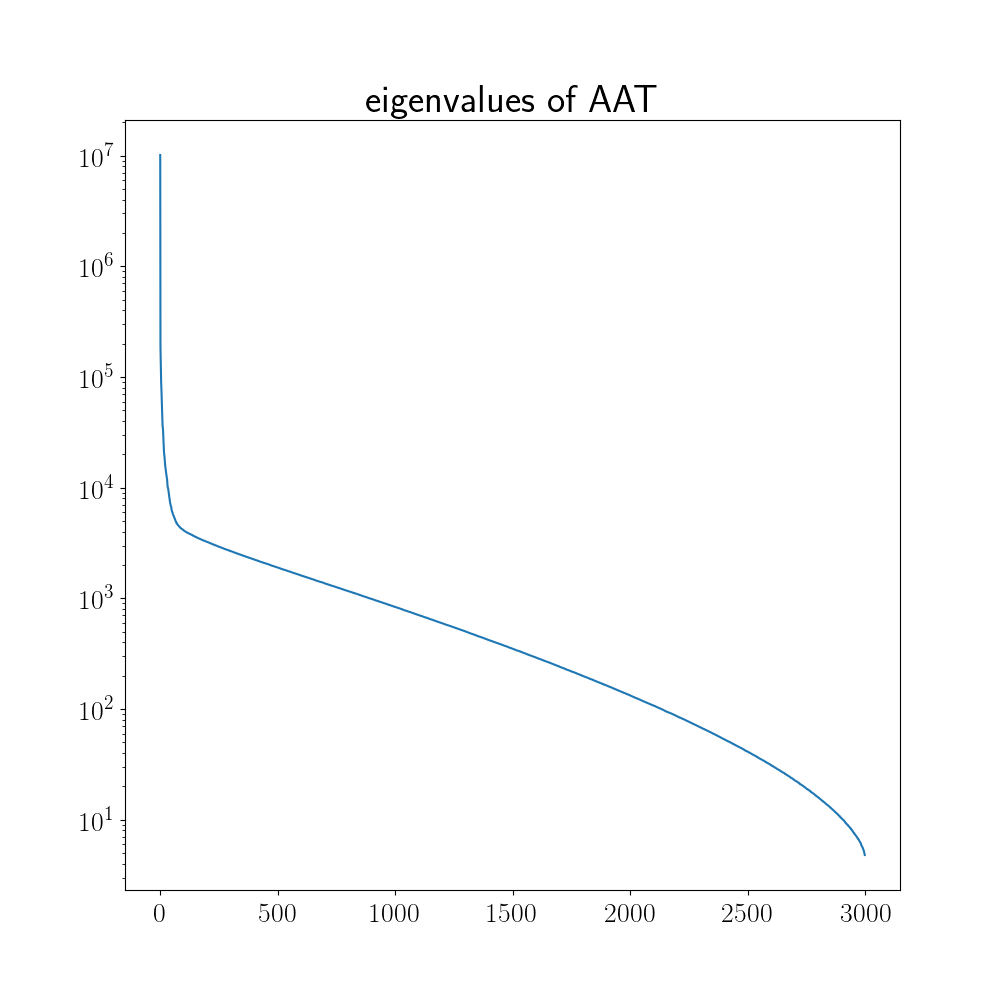}
\end{minipage}
\begin{minipage}[t]{\figsize\textwidth}
\centering
\includegraphics[width=\linewidth]{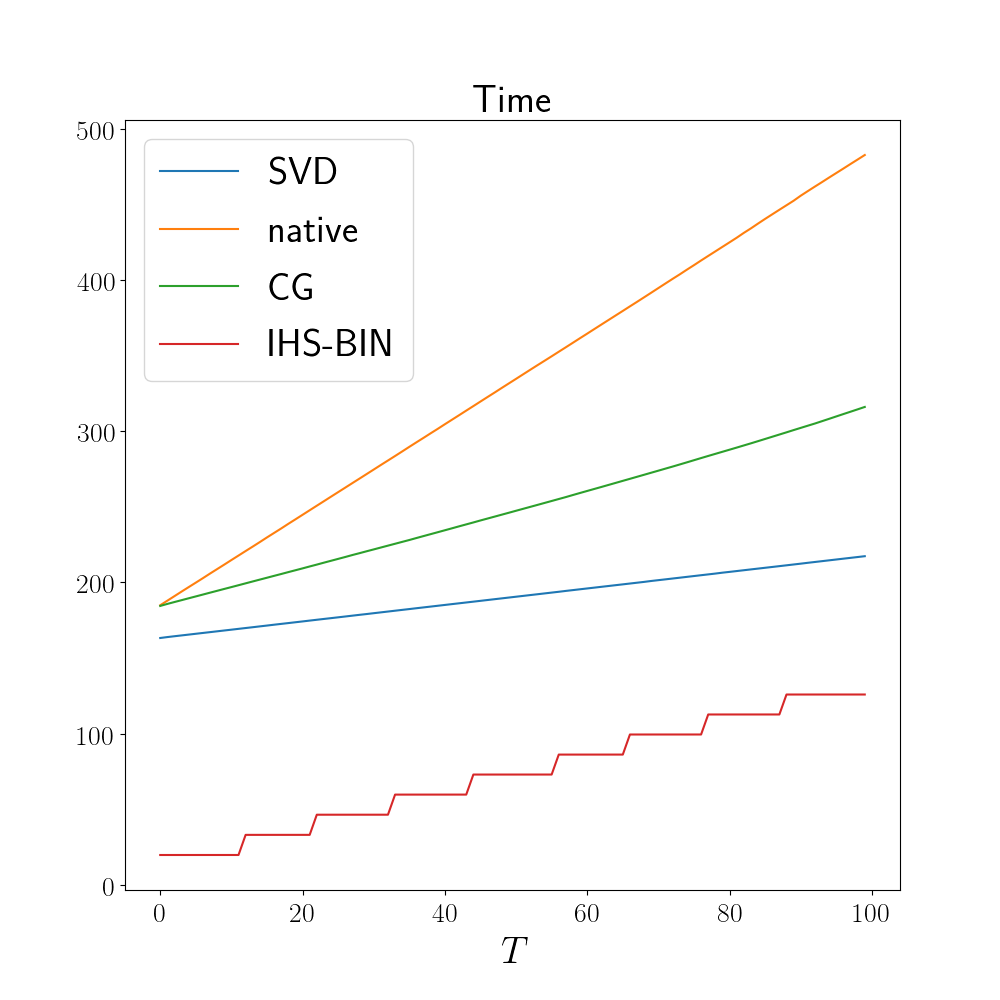}
\end{minipage}
\begin{minipage}[t]{\figsize\textwidth}
\centering
\includegraphics[width=\linewidth]{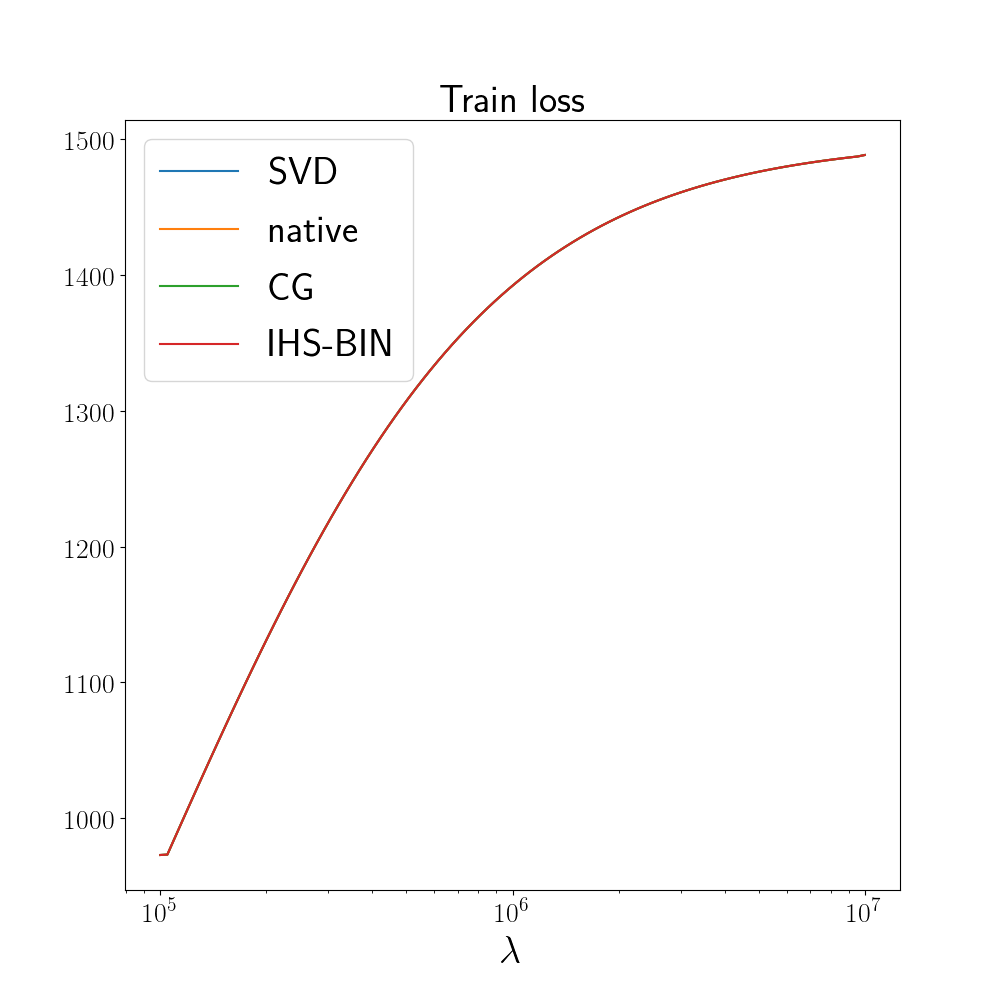}
\end{minipage}
\begin{minipage}[t]{\figsize\textwidth}
\centering
\includegraphics[width=\linewidth]{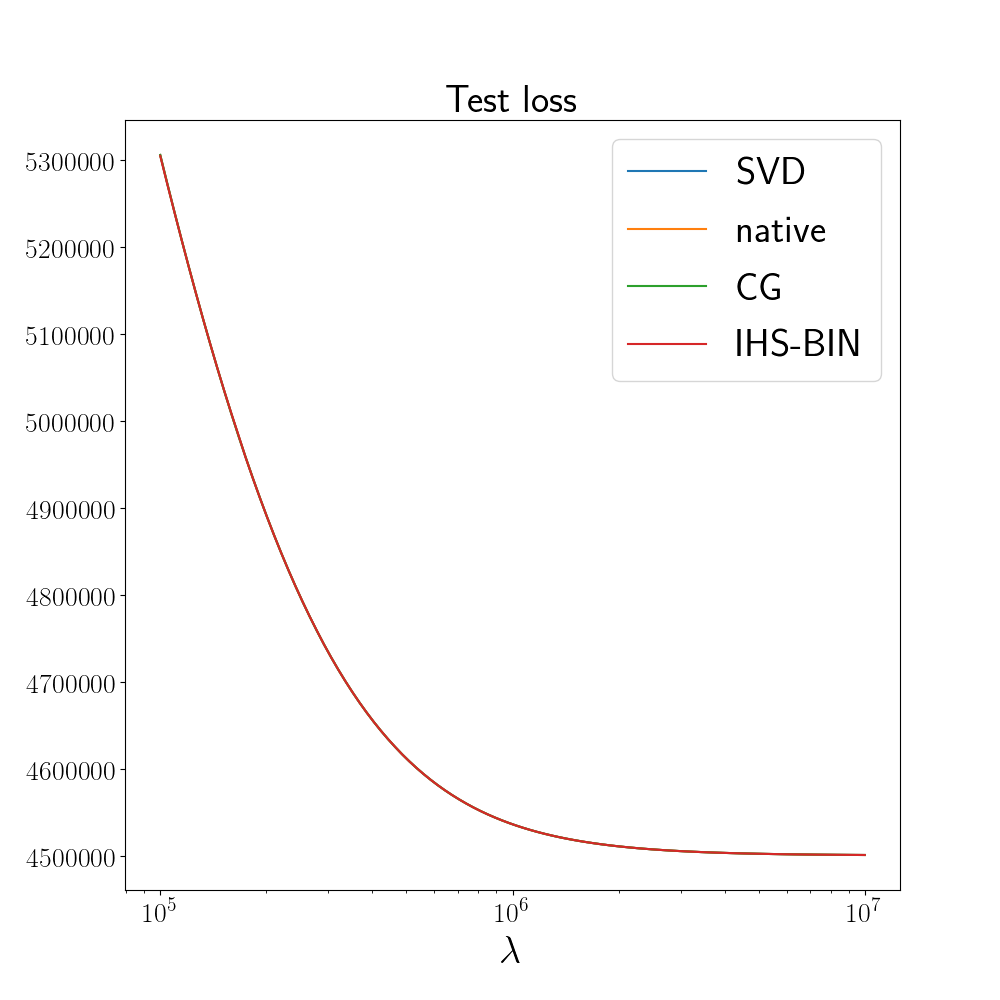}
\end{minipage}
\caption{Training loss, test loss and time. NIPS 2003 Feature Selection Challenge (gisette). $n=3000, d=5000$, $m=800$, $\lambda_\text{min}=10^{5}$. $\lambda_\text{max}=10^{7}$.}
\label{fig:gisette}
\end{figure}

\begin{figure}[!htbp]
\centering
\begin{minipage}[t]{\figsize\textwidth}
\centering
\includegraphics[width=\linewidth]{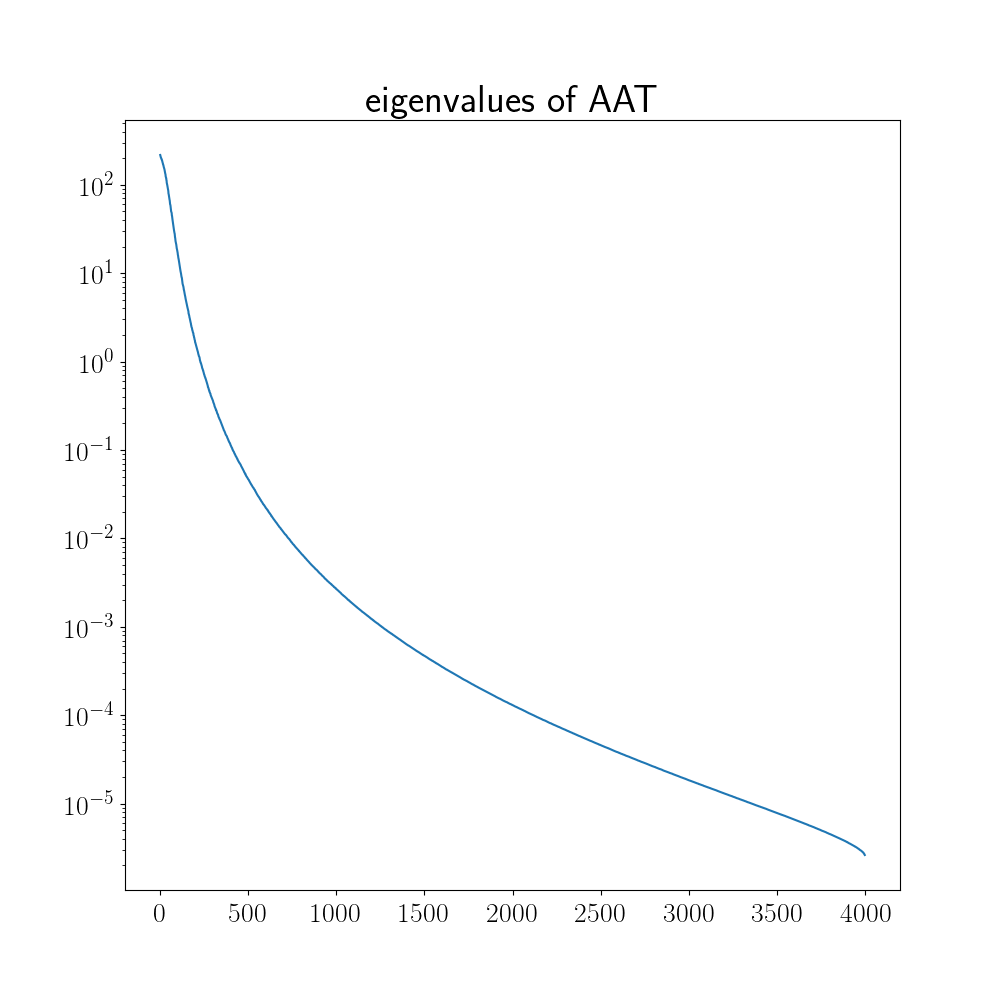}
\end{minipage}
\begin{minipage}[t]{\figsize\textwidth}
\centering
\includegraphics[width=\linewidth]{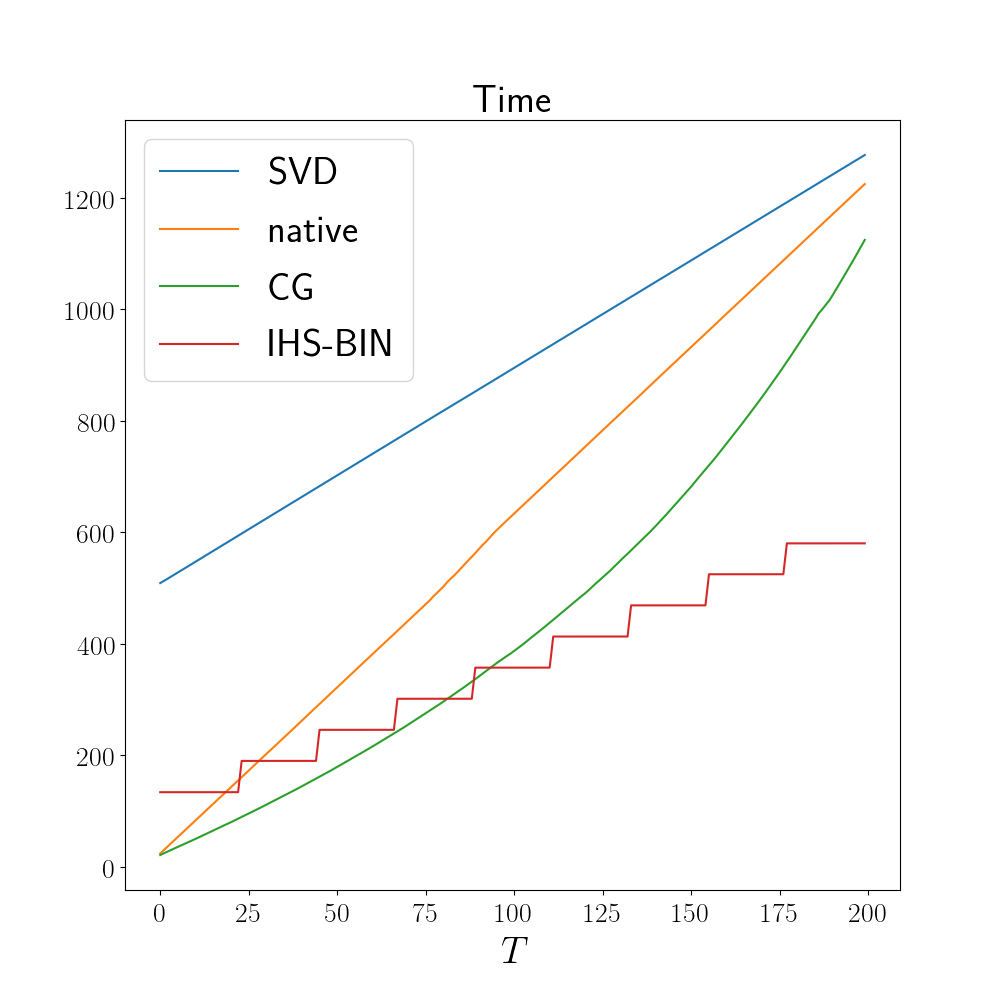}
\end{minipage}
\begin{minipage}[t]{\figsize\textwidth}
\centering
\includegraphics[width=\linewidth]{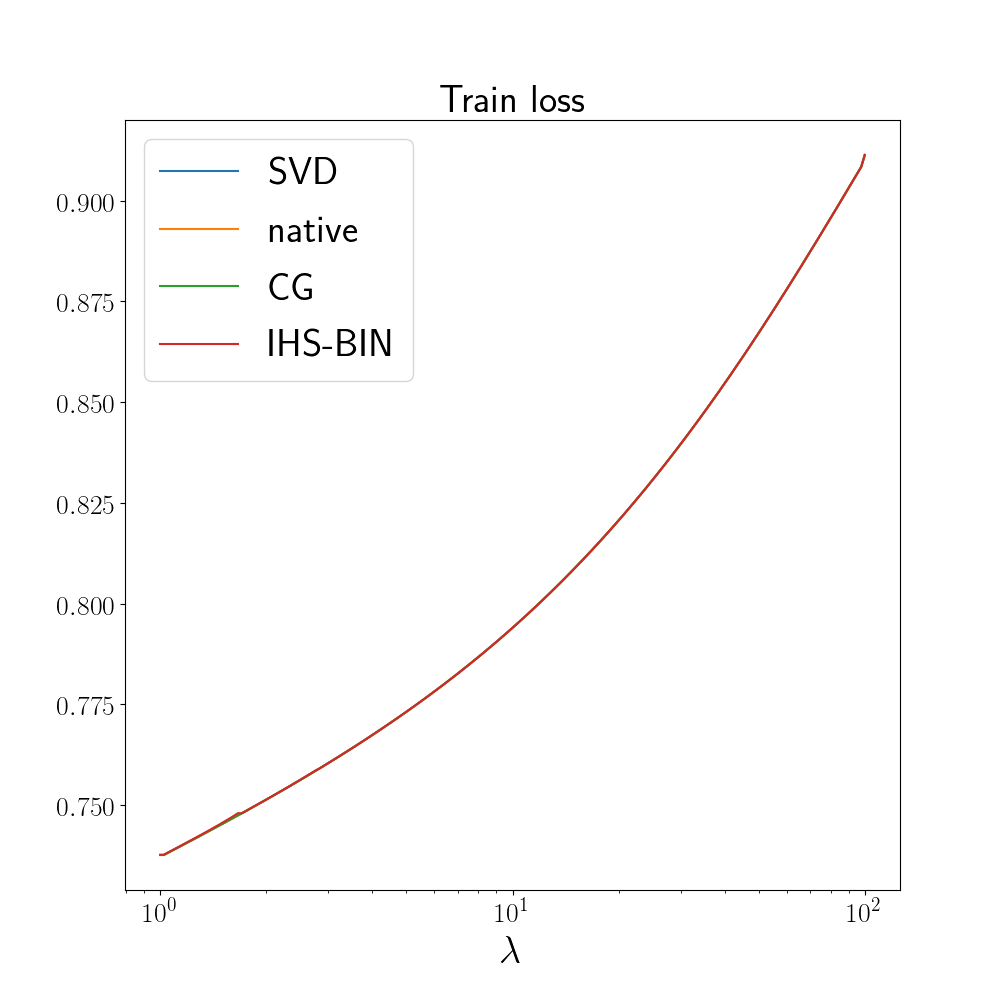}
\end{minipage}
\begin{minipage}[t]{\figsize\textwidth}
\centering
\includegraphics[width=\linewidth]{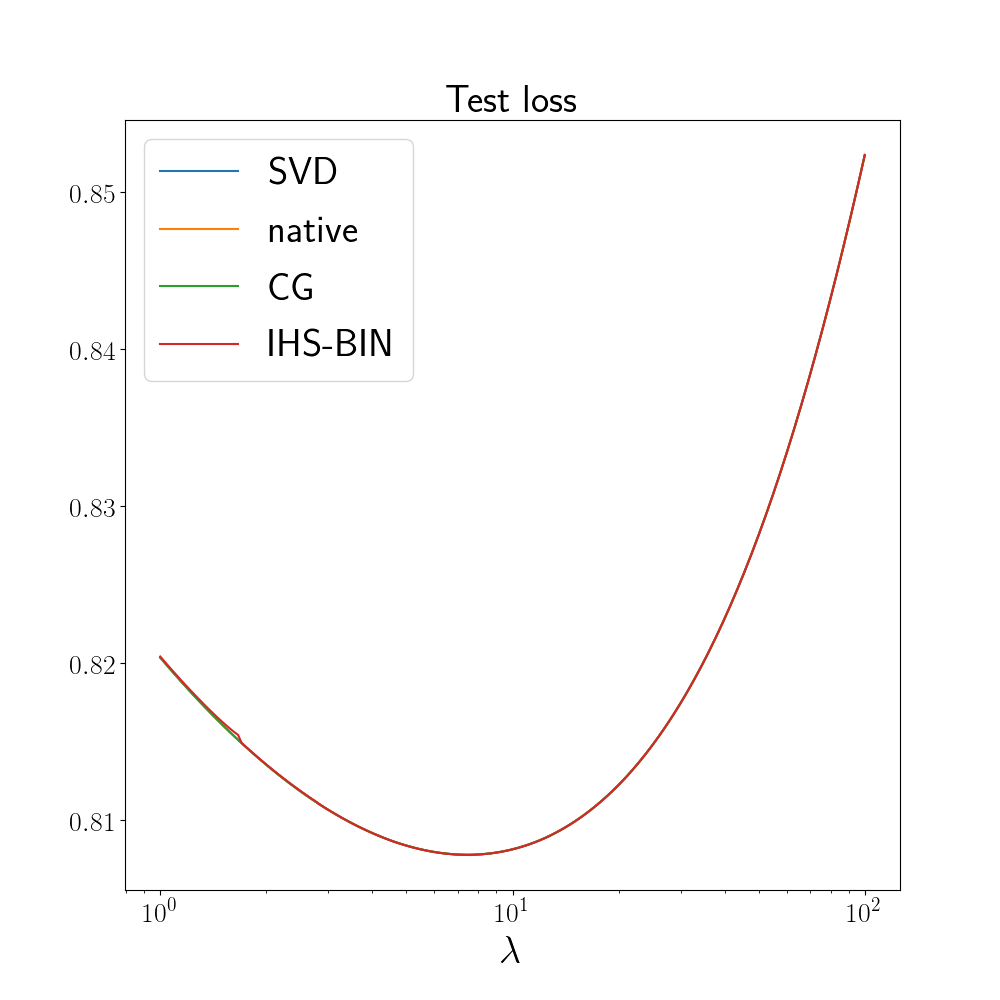}
\end{minipage}
\caption{Training loss, test loss and time. Randomly generated data. $n=4000, d=20000$, $m=2400$. $\sigma=0.02$.  $\lambda_\text{min}=1$. $\lambda_\text{max}=100$.}
\label{fig:random_op}
\end{figure}
\begin{figure}[!htbp]
\centering
\begin{minipage}[t]{\figsize\textwidth}
\centering
\includegraphics[width=\linewidth]{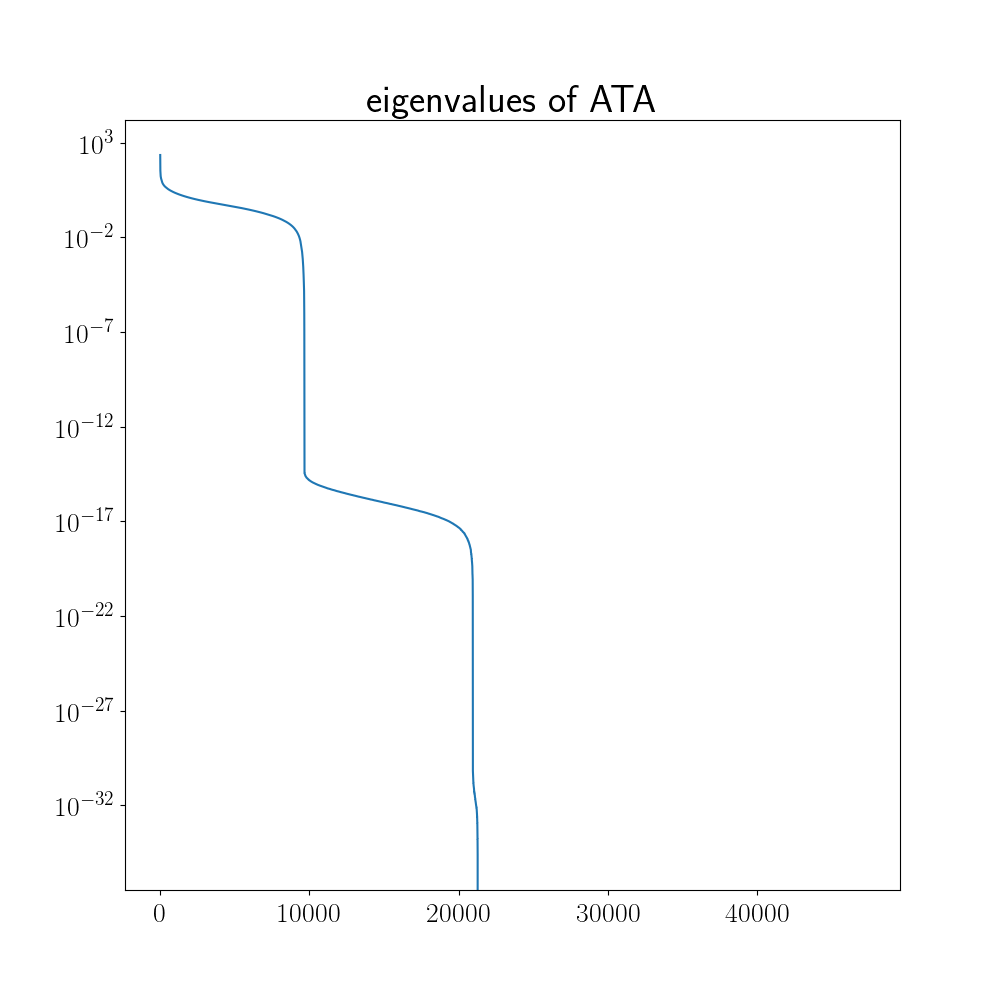}
\end{minipage}
\begin{minipage}[t]{\figsize\textwidth}
\centering
\includegraphics[width=\linewidth]{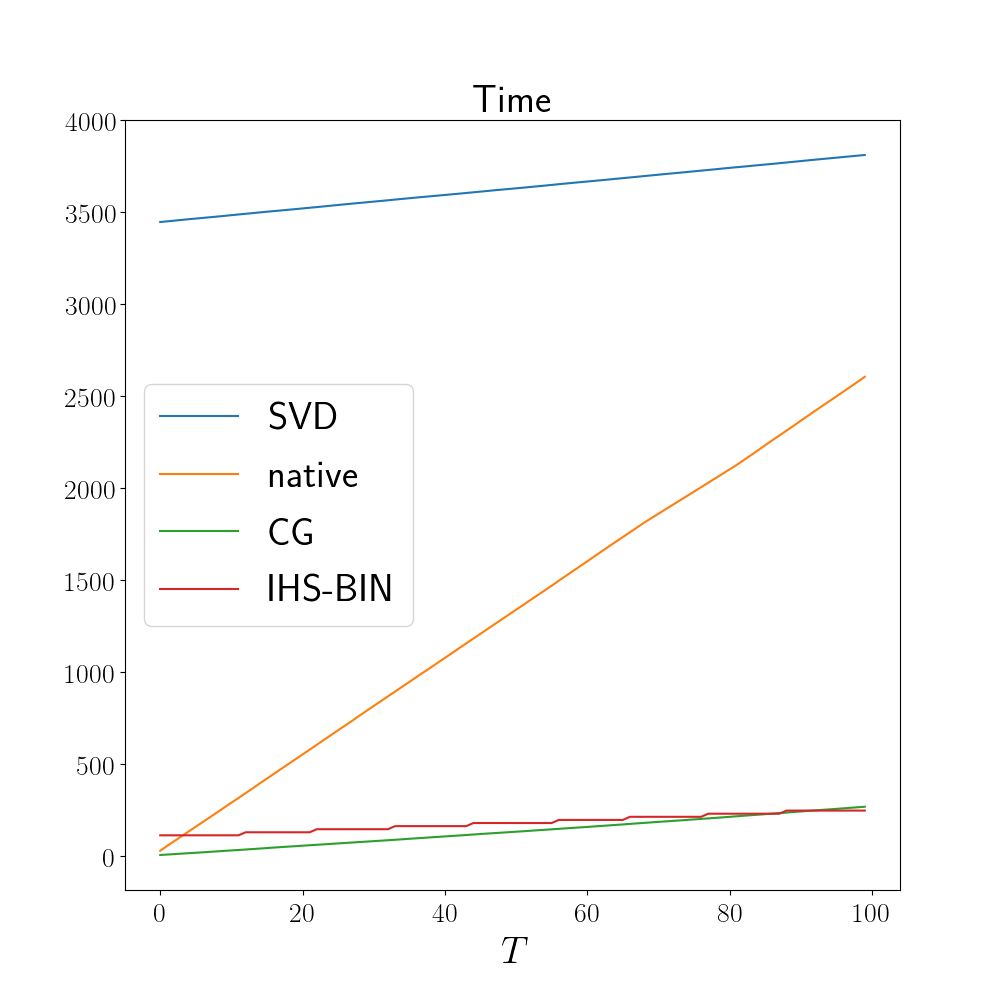}
\end{minipage}
\begin{minipage}[t]{\figsize\textwidth}
\centering
\includegraphics[width=\linewidth]{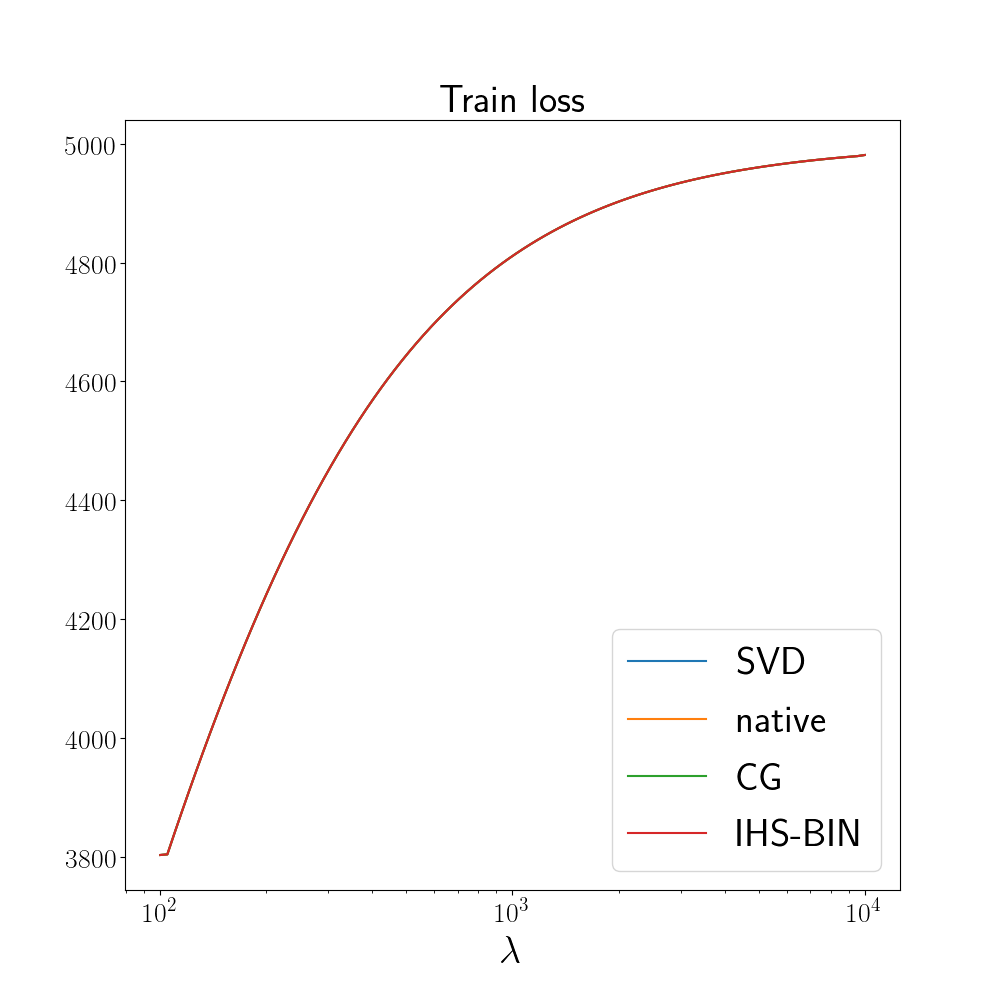}
\end{minipage}
\begin{minipage}[t]{\figsize\textwidth}
\centering
\includegraphics[width=\linewidth]{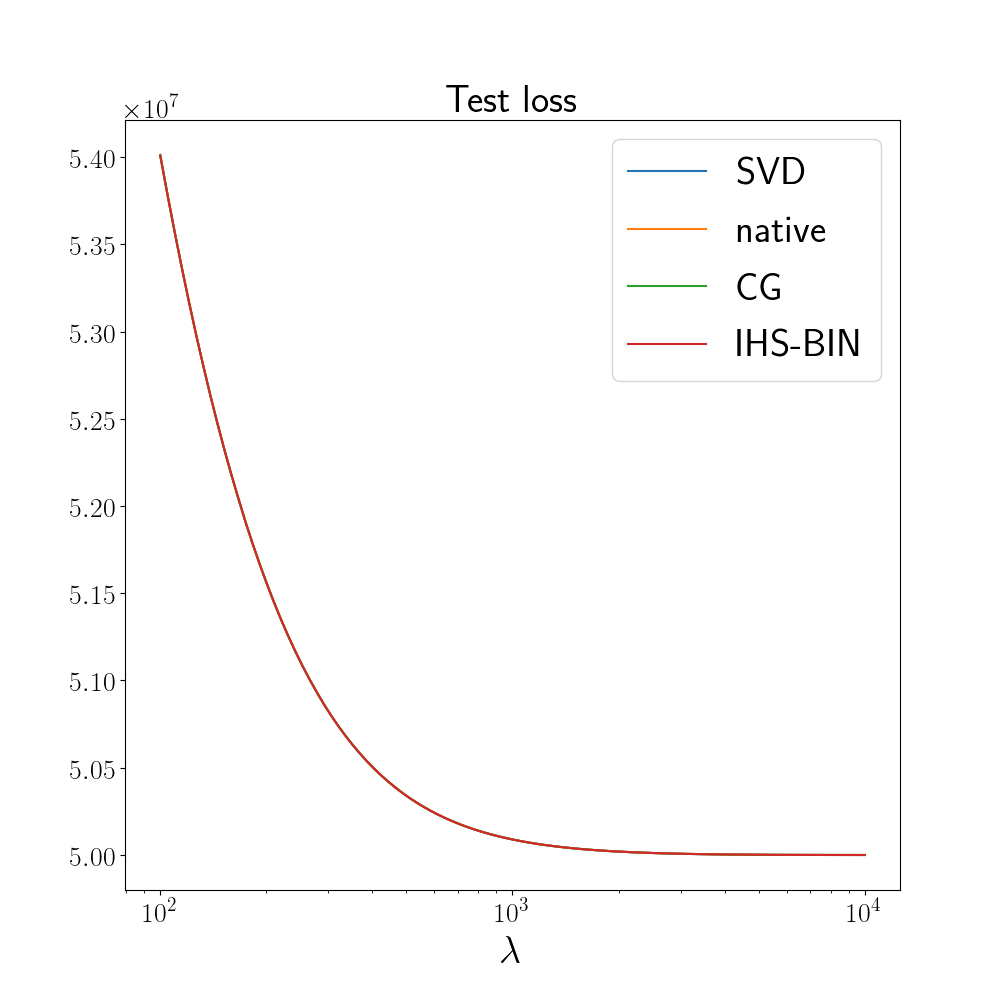}
\end{minipage}
\caption{Training loss, test loss and time. RCV1: A new benchmark collection for text categorization research. $n=10000, d=47236, m=3000$, $\sigma=0.03$. $\lambda_\text{min}=100$. $\lambda_\text{max}=10^4$.}
\label{fig:rcv1}
\end{figure}

\begin{figure}[!htbp]
\centering
\begin{minipage}[t]{\figsize\textwidth}
\centering
\includegraphics[width=\linewidth]{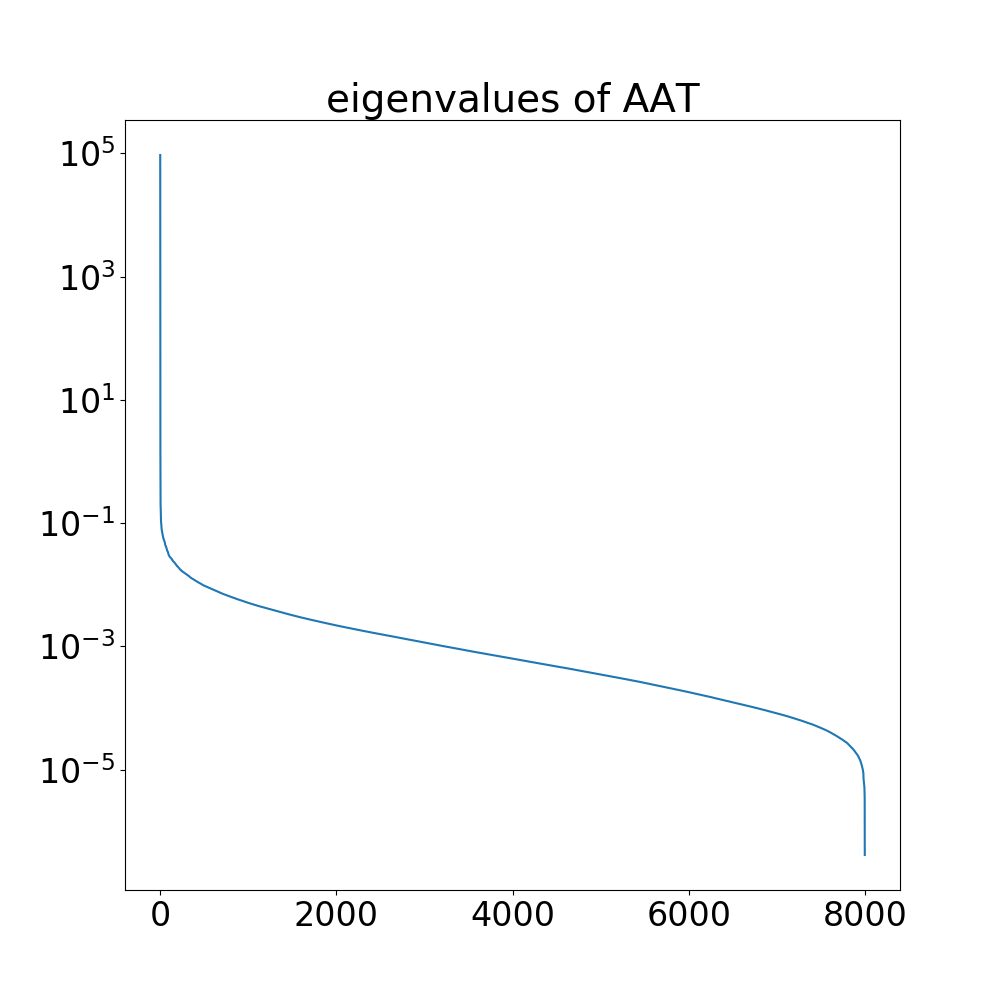}
\end{minipage}
\begin{minipage}[t]{\figsize\textwidth}
\centering
\includegraphics[width=\linewidth]{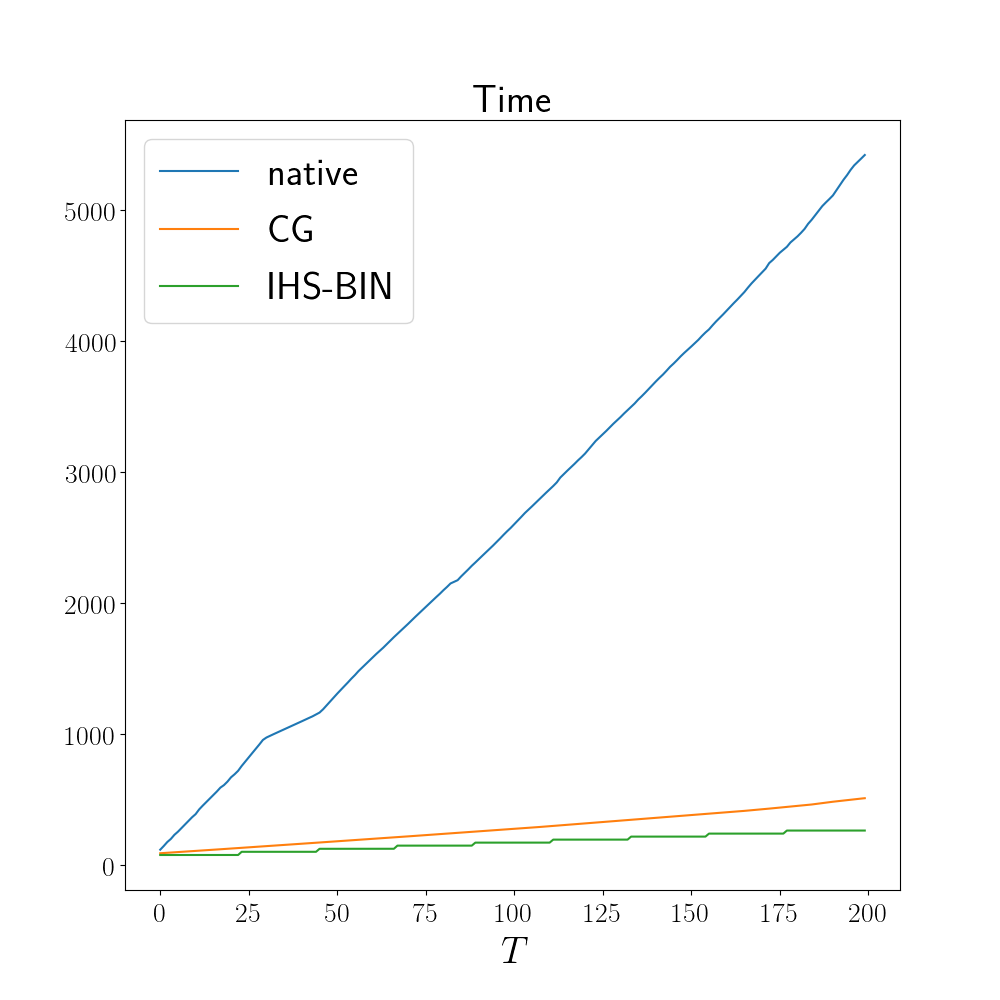}
\end{minipage}
\begin{minipage}[t]{\figsize\textwidth}
\centering
\includegraphics[width=\linewidth]{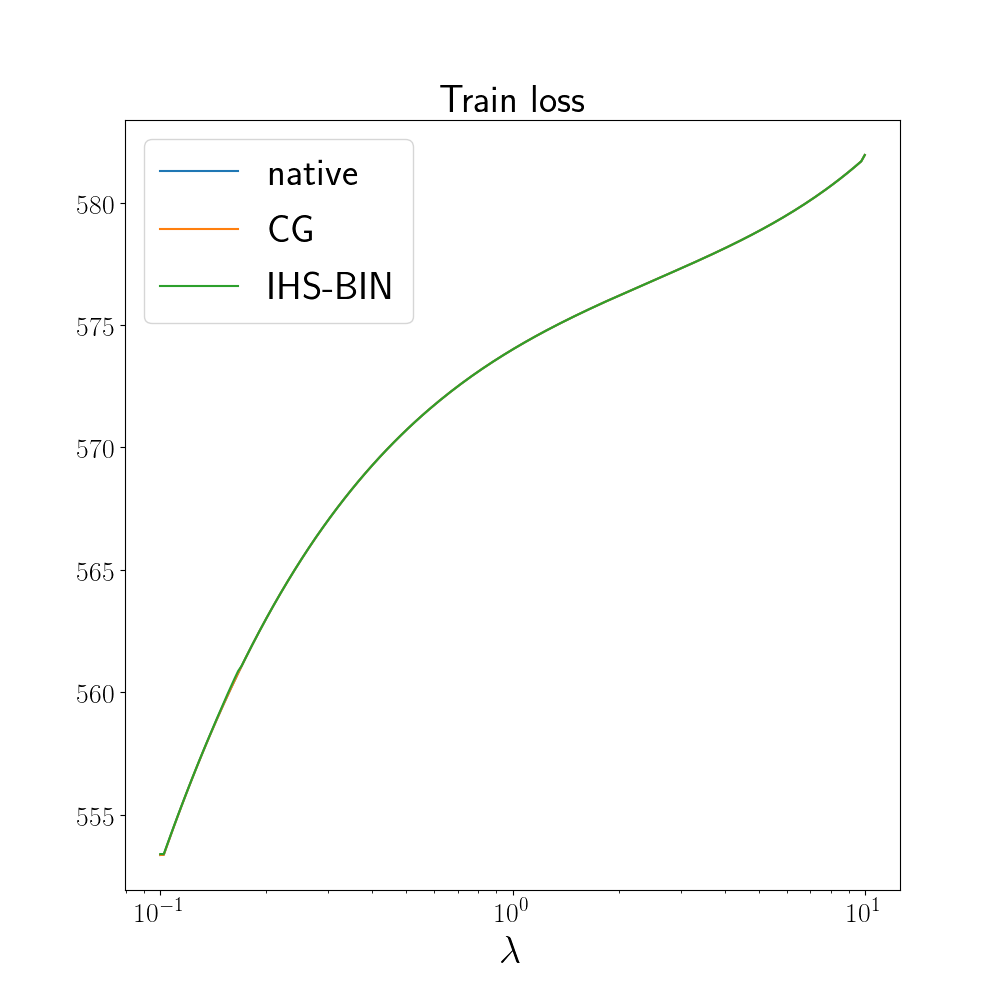}
\end{minipage}
\begin{minipage}[t]{\figsize\textwidth}
\centering
\includegraphics[width=\linewidth]{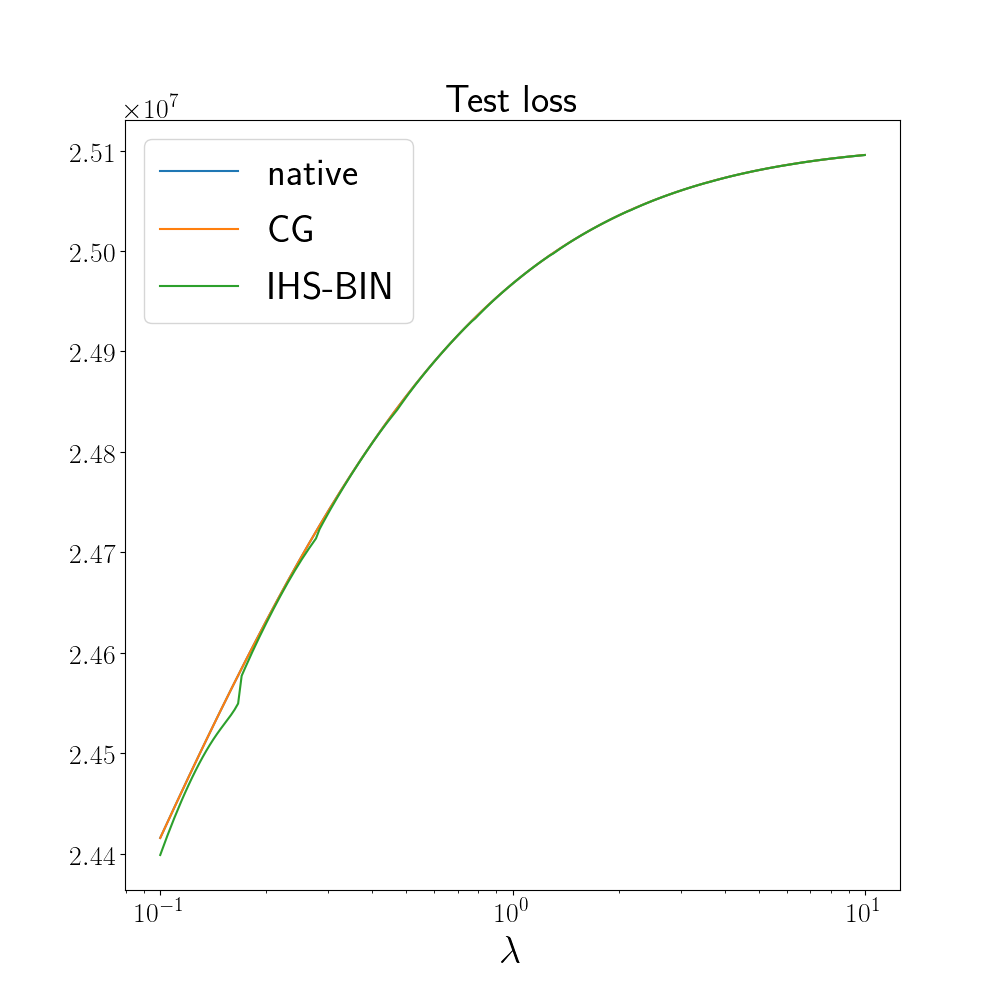}
\end{minipage}
\caption{Training loss, test loss and time. 10-K Corpus (E2006-tfidf). $n=8000, d=150360, m=2000$. $\lambda_\text{min}=0.1$. $\lambda_\text{max}=10$.}
\label{fig:tfidf}
\end{figure}

\section{Conclusion}
We presented IHS-BIN for rapidly computing the entire ridge regularization path. The algorithm is based on analyzing the gradient descent regularization path and accelerating convergence via randomized sketching. Our method improves the state-of-the-art computational complexity of obtaining the solution of ridge regression for $T$ values of the regularization parameter from $\mcO(Td^2)$ or $\mcO(T\mathrm{nnz}(A))$ to $\mcO(Td)$, when $T$ is large. The numerical experiments also demonstrate that IHS-BIN is significantly faster than other solvers, especially for large-scale problems. Our method also leverages the low effective dimensionality of real datasets, which can be used to reduce the sketching dimension. We also investigated adaptively picking the sketch dimension based on the progress of the algorithm. We believe that our algorithm will be quite effective in automatically tuning the regularization parameters of linear models. Moreover, our method can be used in transfer learning and deep feature embedding for training a final linear layer and tuning the regularization parameter efficiently. One potential limitation of the proposed approach is that for medium-size data matrices with high effective dimensions, direct methods such as SVD can be more effective. 
\bmhead{Acknowledgments} 
This work is partially supported by an Army Research Office Early Career Award, the National Science Foundation under grants IIS-183817, ECCS-2037304 and DMS-2134248, ACCESS – AI Chip Center for Emerging Smart Systems, sponsored by InnoHK funding, Hong Kong SAR and a Precourt Institute seed grant.

\bibliography{Newton}


\end{document}